\documentclass[review]{article}
\usepackage{amsfonts}
\usepackage{graphicx}
\usepackage{epstopdf}
\usepackage{amsmath}
\usepackage{graphicx}
\usepackage{amssymb}
\usepackage{mathtools}
\usepackage{xcolor}
\usepackage{ragged2e}
\usepackage{stmaryrd}
\usepackage{bm}
\usepackage{float}
\usepackage{cite}
\usepackage{subcaption}
\usepackage{todonotes}
\usepackage[normalem]{ulem}
\usepackage{hyperref}
\usepackage{enumerate}
\usepackage{amsthm}
\usepackage{authblk}

\newcommand{\stkout}[1]{\ifmmode\text{\sout{\ensuremath{#1}}}\else\sout{#1}\fi}

\numberwithin{equation}{section}
\newtheorem{theorem}{Theorem}[section]
\newtheorem{lemma}[theorem]{Lemma}
\newtheorem{prop}[theorem]{Proposition}
\newtheorem{assum}[theorem]{Assumption}

\newtheorem{rem}[theorem]{Remark}

\newcommand{\vertiii}[1]{{\left\vert\kern-0.25ex\left\vert\kern-0.25ex\left\vert #1 
    \right\vert\kern-0.25ex\right\vert\kern-0.25ex\right\vert}}

\newcommand{\norm}[1]{\left\|#1\right\|_{m,\beta}}
\newcommand{\expect}[1]{\mathbb{E}\left[#1\right]}
\providecommand{\keywords}[1]{\textbf{\text{Key words.}} #1}
\providecommand{\AMS}[1]{\textbf{\text{AMS subject classifications.}} #1}

\begin{document}

\title{Strong convergence of  a  Verlet integrator for the semi-linear stochastic wave equation
}

\author[1]{Lehel Banjai\thanks{l.banjai@hw.ac.uk}}
\author[2,1]{Gabriel Lord\thanks{gabriel.lord@ru.nl}}
\author[1]{Jeta Molla\thanks{jm188@hw.ac.uk}}
\affil[1]{Maxwell  Institute  for  Mathematical  Sciences,  School  of  Mathematical  \&  Computer  Sciences, Heriot-Watt University, Edinburgh, EH14 4AS, UK}
\affil[2]{Department of Mathematics, IMAPP, Radboud University, Nijmegen, The Netherlands}

\maketitle

\begin{abstract}
The full discretization of the semi-linear stochastic wave equation is considered. The discontinuous Galerkin finite element method is used in space and analyzed in a semigroup framework, and an explicit  stochastic position Verlet  scheme is used for the temporal approximation. We study the stability under a CFL condition and prove optimal strong convergence rates of the fully discrete scheme. Numerical experiments illustrate our theoretical results. Further, we analyze and  bound  the expected energy and numerically show excellent agreement with the energy of the exact solution.  
\end{abstract}

\keywords{
 semi-linear stochastic wave equation,  stochastic Verlet integration,  strong convergence, discontinuous Galerkin finite element method, stability, energy conservation}

\AMS{
60H15, 60H35, 65C30, 65C20,  65M60}

\section{Introduction} 
We study the semi-linear stochastic wave equation  driven by  additive noise
\begin{equation} \label{wave}
\begin{aligned}
 &{\rm d}\dot{u}=\nabla\cdot ({\bf D} \nabla u)  {\rm d}t+f(u){\rm d}t +{\rm d}W  &&{\rm in}\;\; {\cal D}\times(0,T),\\
&\; u= 0  &&{\rm on } \; \;\partial {\cal D} \times(0,T),\\
&\;  u(\cdot,0) = u_0,\; \dot{u}(\cdot,0)= v_0 & &{\rm in} \;\; {\cal D},
\end{aligned}
\end{equation}
where ${\cal D}\subset \mathbb{R}^d,$ $d=1,2,3,$ is a bounded convex polytopal domain with boundary $\partial {\cal D}$ and $\dot{u} = \partial_t u$ stands for the time derivative, 
$f$ is a globally Lipschitz nonlinear function, and $\{ W(t)\}_{t\geq0}$ is a $Q$-Wiener process with respect to a normal filtration $\{{\cal F}_t \}_{t\geq0}$ on a filtered probability space $\left(\Omega, {\cal F}, \mathbb{P}, \{ {\cal F}_t\}_{t\geq0} \right)$. We give our assumptions on the noise and $f$ in detail in Section \ref{sec:not}.
The initial data $u_0$ and $v_0$ are ${\cal F}_0$-measurable variables.
We assume that ${\bf D}=(d_{ij}(x))_{1\leq i,j\leq d}$ is a symmetric uniformly positive definite matrix that  satisfies the bounds  
\begin{equation} \label{eq:wavespeed}
0< d_{\min}\leq \zeta^T{\mathbf{D}(x)} \zeta \leq d_{\max} < \infty, \quad \text{for all }\zeta \in \mathbb{R}^d, \;|\zeta| = 1,\; x\in \cal D.
\end{equation}
The stochastic wave equation is of fundamental importance in various   applications.  For instance, the motion of a strand of DNA floating in a liquid \cite{dalang};  the dynamics of the primary current density vector field within the grey matter of the human brain \cite{galka2008}; or the vibration of a string under the action of stochastic forces \cite{orshinger}.  Our motivating example arises from the sound propagation in the sea, critical for marine reserves and conservation of species \cite{ikpekha2014, etter2012}. In particular we note that in the marine environment the noise is typically only over a finite range of frequencies and also that often complex computational domains are of interest where sound waves may interact with the shore. The complex geometry motivates the use of the dG method in space and we pay attention not only to space-time rough noise forcing but also to more regular noise in space.

One  advantage the dG method offers over the standard continuous finite element methods is that the mass matrix has a block diagonal structure; it can therefore be inverted at  a very low computational cost.  Hence, the dG method leads to efficient time integration when combined with an explicit time-stepping scheme. 
In the present paper, we propose such full discretization of \eqref{wave} by applying the dG finite element method in space \cite{Arnold1982} and  a stochastic extension of  the explicit position Verlet time-stepping method in time \cite{casas2016, serna2014}. 
The position Verlet scheme is a variant of the St\"{o}rmer-Verlet/leapfrog method and it shares their important geometric properties such as symplecticity. We refer the reader to \cite{hairer2003} for a review on the St\"{o}rmer-Verlet/leapfrog integrators.
Stochastic extensions of the position Verlet or St\"{o}rmer-Verlet/leapfrog time-stepping schemes have  been used for second-order ordinary stochastic differential equations, see e.g. \cite{burrage2007, jensen2013, cohen2019, farago2019} and references therein.

In recent years, strong approximations of  stochastic wave equations  have been studied by many authors \cite{anton2015,cohen2013,Kovacs2010, yin2007, schurz2008, wang2014, wang20142, cui2019, walsh2006, sard2006, cohen2015}.
We first comment on work analyzing the spatial discretization of  stochastic wave equations. Strong convergence estimates for the  continuous finite element approximation of  the linear stochastic wave equation and of (\ref{wave}) with multiplicative noise  were derived in \cite{Kovacs2010} and \cite{anton2015}, respectively. In  \cite{schurz2008}, \cite{wang2014}, and \cite{yin2007} the spectral Galerkin method for one-dimensional semi-linear stochastic wave equations driven by additive noise was used in space and strong convergence rates were proved. In addition, in \cite{wang20142} the  stochastic wave equation with Lipschitz continuous nonlinearity and multiplicative noise is discretized by the spectral method in space. Further, the recent preprint \cite{cui2019} studies the spectral Galerkin approximation of the  stochastic wave equation with polynomial nonlinearity in $\mathbb{R}^d$, $d=1,2,3 ,$ and analyzes strong convergence in  $L^p{(\Omega)}$.  Finite difference spatial discretizations for one-dimensional  stochastic wave equations were employed in \cite{sard2006, walsh2006, cohen2015}.
Secondly we comment on time stepping schemes for  stochastic wave equations. Stochastic trigonometric methods have been used for the temporal approximation of various types of  stochastic wave equations, see e.g. \cite{anton2015, cohen2013,cohen2015,wang20142} and improved convergence rates  were derived in \cite{wang2014} by using linear functionals of the noise as in \cite{jentzen2009}. Strong convergence rates for the St\"{o}rmer-Verlet/leapfrog time-stepping scheme applied to the one-dimensional stochastic wave equation  and one-step $I$-stable time-stepping methods can be found in \cite{walsh2006} and \cite{Kovacs2013}, respectively. 
Finally, the recent preprint \cite{cui2019} analyzes the strong convergence of a splitting average vector field time-stepping method. 

Concerning the dG finite element approximation of parabolic stochastic partial differential equations (SPDEs) we refer the reader to \cite{li2020, chen2018}. In \cite{chen2018} the exact solution was represented in terms of Green's functions and strong convergence estimates were derived for dG approximation to the linear Cahn-Hilliard equation. In the recent publication \cite{li2020} a local dG method is analyzed for nonlinear parabolic SPDEs with multiplicative noise and strong spatial convergence rates are derived.

Our main aim here is to prove strong convergence to the mild solution of the stochastic position Verlet method (SVM) with a dG discretization in space applied to \eqref{wave},  by introducing a discrete norm, under a CFL condition, which is preserved by the time integrator. We note that the same analytical tool could be used to analyze the stability of  the St\"{o}rmer-Verlet/leapfrog method.
To the best of our knowledge, the semigroup approach to the dG formulation that we adopt here  has not been considered elsewhere. The martingale approach in  \cite{walsh2006} for the St\"{o}rmer-Verlet/leapfrog scheme applies only to the one-dimensional case, while our analysis  enables us to obtain optimal error bounds for both the displacement and the velocity in multiple dimensions.

The paper is organized as follows. We introduce some notation, introduce
our assumptions on the noise and the nonlinearity, and rewrite \eqref{wave} as a first order system
in Section \ref{sec:not}. In Section \ref{sec:spacedis} we discuss the existence and  the uniqueness of the dG finite element approximation of the stochastic equation (\ref{wave}) and we extend the results of \cite{anton2015} on the strong convergence estimates for the spatial discretization of our problem. Stability and strong convergence analysis of SVM is considered in Section \ref{sec:fullydiscrete}.  In Section \ref{sec:energy} we state energy results of the full discretization of our problem and in Section \ref{sec:numexp} we present numerical experiments in order  to demonstrate the theory and the efficiency of our discretization.

\section{Preliminaries and notation}
\label{sec:not}
Let $U$ and $H$ be separable Hilbert spaces with norms $\| \cdot \|_U$ and $\|\cdot\|_H$, respectively. We denote the space of linear operators from $U$ to $H$ by ${\cal L}(U, H)$, and we let ${\cal L}_2(U,H)$ be the set of Hilbert-Schmidt operators with norm 
\begin{equation*} 
\|R\|_{{\cal L}_2(U,H)}:=\left( \sum\limits_{k=1}^{\infty}\|Re_k\|_H^2\right)^{1/2},
\end{equation*}
where $\{e_k \}_{k=1}^{\infty}$ is an arbitrary orthonormal basis of $U$. If $H=U$, then we write ${\cal L}(U)={\cal L}_2(U, U)$ and HS=${\cal L}_2(U, U)$.
 Let $L_2(\Omega, H)$ be the space of $H$-valued square integrable random variables with norm
\begin{equation*} 
\| v\|_{L_2(\Omega, H)}:=\mathbb{E}\left[ \|v\|^2_H\right]^{1/2}.
\end{equation*}
Let $Q\in{\cal L}(U)$ be a self-adjoint, positive semidefinite operator. The driving stochastic process $W(t)$ in (\ref{wave}) is a $U$-valued $Q$-Wiener process with respect to the filtration $\{{\cal F}_t\}_{t\geq0}$ and has the orthogonal expansion \cite[Section 10.2]{gabriel2014}
\begin{equation} \label{eq:wiener}
W(t)=\sum\limits_{j=1}^{\infty}q_j^{1/2}\beta_j(t)\psi_j,
\end{equation}
where $\{\psi_j\}_{j=1}^{\infty}$ are  orthonormal eigenfunctions of $Q$ with corresponding eigenvalues $q_j \geq 0$ and $\{\beta_j(t) \}_{j=1}^{\infty}$ are real-valued mutually independent standard Brownian motions. Suppose that $\{\Phi(s)\}_{0\leq s \leq t}\in {\cal L}(U,H)$ and that  
\begin{equation*} 
\int_{0}^{t}\|\Phi(s)Q^{1/2}\|_{\rm HS}{\rm d}s< \infty,
\end{equation*}
then the stochastic integral $\int_{0}^{t}\Phi(s){\rm d}W(s)$ is well defined in $L_2(\Omega,H)$ and we have It\^{o}'s isometry, see  \cite[Section 10.2]{gabriel2014}, 
\begin{equation} \label{eq:itosisometry}
\left\|\int_{0}^{t}\Phi(s){\rm d}W(s)\right\|_{L_2(\Omega,H)} = \int_{0}^{t}\|\Phi(s)Q^{1/2}\|_{{\cal L}_2(U,H)}{\rm d}s.
\end{equation} 
Let us introduce the spaces and norms that we use to describe the spatial regularity of functions.  Let $\Lambda =-\nabla \cdot ({\bf D}\nabla u)$ denote the  linear  operator $\Lambda  \colon D(\Lambda) \rightarrow L_2({\cal D}) $ with $D(\Lambda)=H^2({\cal D})\cap H^1_0({\cal D})$ and let $L_2({\cal D})$ have the usual inner product $(\cdot, \cdot)$ and norm $\| \cdot \|$.  Further, we define the following spaces  
\begin{equation*} 
\dot{H}^\alpha = D(\Lambda^{\alpha/2}), \quad \| v \|_\alpha = \| \Lambda^{\alpha/2}v\|=\left( \sum\limits_{j=0}^{\infty}\lambda_j^\alpha(v,\phi_j)^2\right)^{1/2}, \quad \alpha\in\mathbb{R}, \, v\in \dot{H}^\alpha,
\end{equation*}
where $\{(\lambda_j, \phi_j )\}_{j=1}^{\infty}$ are the eigenpairs of $\Lambda$ with orthonormal eigenvectors.  We also introduce a product Hilbert space with accompanying norm
\begin{equation*} 
{\cal H}^\alpha = \dot{H}^\alpha\times \dot{H}^{\alpha-1}, \quad \vertiii{v}_\alpha^2 =\| v_1\|_\alpha^2 +\| v_2\|_{\alpha-1}^2, \quad \alpha\in\mathbb{R},\, v\in {\cal H}^\alpha.
\end{equation*}

 To study (\ref{wave}) as an abstract stochastic differential equation on the Hilbert space ${\cal H}^1$, we use the notation $u_1 := u$ and  $u_2:=\dot{u}_1=\dot{u}$, and rewrite (\ref{wave}) as follows
\begin{equation} \label{wavesde}
{\rm d}U(t) = AU(t){\rm d}t +F(U(t))dt+ B{\rm d}W, \quad U(0)=U_0, \quad t\in(0,T),
\end{equation} 
where 
\begin{equation}\label{eq:defns_1storder}
 U=\left[ \begin{matrix}  u_{1} \\ u_{2}
\end{matrix}\right], \quad A= \begin{bmatrix} 0 &I \\- \Lambda & 0 \end{bmatrix}, \quad F(U(t))= \left[\begin{matrix} 0\\ f(u_1(t))\end{matrix}\right],\quad B = \left[\begin{matrix} 0\\ I\end{matrix}\right], \quad U_0 = \left[\begin{matrix} u_0\\ v_0\end{matrix}\right].
\end{equation}
The operator $A$ with 
$
D(A) = \left( H^2({\cal D})\cap H^1_0({\cal D})\right) \times H^1_0({\cal D})
$ is the generator of a strongly continuous semigroup ($C_0$-semigroup) $E(t)=e^{tA}$ on $ H^1_0 \times L_2({\cal D})$ and 
\begin{equation} \label{semigroup}
E(t)=e^{tA}=\begin{bmatrix}
C(t) & \Lambda^{-1/2}S(t)\\
-\Lambda^{1/2}S(t) & C(t)
\end{bmatrix},
\end{equation}
where $C(t)=\cos\left(t\Lambda^{1/2}\right)$ and $S(t)=\sin\left(t\Lambda^{1/2}\right)$. For more detail we refer the reader to \cite[Appendix A]{daprato2014} or \cite[Chapter 10.3]{brezis2010}.

We require that the  $Q$-Wiener process $W(t)$ in (\ref{eq:wiener}) satisfies the following assumption. 
\begin{assum} 
\label{assum:noise}
 The $Q$-Wiener process $W(t)$  (\ref{eq:wiener})  takes values in $\dot H^{\beta -1}$ for a fixed $\beta\geq 0$, i.e., $\mathbb{E} \left[\| W(t)\|^2_{\beta-1} \right]< \infty$.
\end{assum}
 Assumption \ref{assum:noise} is equivalent to requiring that  $\| \Lambda^{(\beta - 1)/2}Q^{1/2}\|^2_{\rm HS}<\infty $ for a fixed $\beta\geq 0$ since 
\begin{equation} \label{eq:noisecond}
\mathbb{E} \left[\| W(t)\|^2_{\beta-1} \right] = \sum\limits_{j=1}^{\infty}\lambda_j^{\beta-1}q_j \mathbb{E}[\beta_j(t)^2]=t \sum\limits_{j=1}^{\infty}\lambda_j^{\beta-1}q_j=t\| \Lambda^{(\beta - 1)/2}Q^{1/2}\|^2_{\rm HS}.
\end{equation}
Hence, $\| \Lambda^{(\beta - 1)/2}Q^{1/2}\|^2_{\rm HS}<\infty $ if and only if  $\mathbb{E} \left[\| W(t)\|^2_{\beta-1} \right]<\infty$ for any finite $t$.

We also assume that the function $f:L_2({\cal D})\rightarrow L_2({\cal D})$ satisfies 
\begin{equation} \label{eq:assumf}
\begin{aligned} 
&\|f(u)-f(v)\|\leq C\|u -v\|, & &\text{for }u,v\in L_2({\cal D}),\\
&\|f(u)\|_\gamma \leq C(1+\|u\|_\gamma), & &\text{for }  u\in \dot{H}^\gamma \text{ and }\gamma \geq 0.
\end{aligned}
\end{equation}

The existence and uniqueness of the mild solution of the  stochastic wave equation (\ref{wave}) is discussed in \cite[Theorem 7.4]{daprato2014} and the spatial Sobolev regularity of the solution  is discussed in \cite[Proposition 3.1]{wang20142}.
\begin{theorem}\label{thm:cont_exist}
Assume that the function f satisfies (\ref{eq:assumf}) and that $\| U_0\|_{L_2(\Omega,{\cal H}^\beta)}<\infty$, for some $\beta\geq 0$. Then under Assumption \ref{assum:noise} the stochastic wave equation (\ref{wavesde}) has a unique mild solution,  given  by 
\begin{equation} \label{mildsolution}
U(t) = E(t)U_0 + \int_0^t E(t-s)F(U(s)){\rm d}s + \int_0^t E(t-s)B{\rm d}W(s).
\end{equation}
Additionally,  there exists a constant $C$ depending on $T$ and  $\|\Lambda^{(\beta-1)/2}Q^{1/2}\|_{\rm HS}$ such that 
\begin{equation} 
\| U(t)\|_{L_2(\Omega, {\cal H}^\beta)}\leq C(\| U_0\|_{L_2(\Omega, {\cal H}^\beta)} +1), \quad 0\leq t\leq T.
\end{equation}
\end{theorem}

\section{Spatial semi-discretization}
\label{sec:spacedis}
In this section, we discretize (\ref{wave}) by using the interior penalty dG  finite element method \cite{Arnold1982} in space and provide existence and uniqueness of the dG semi-discrete formulation of (\ref{wavesde}) in a semigroup framework. For an introduction to dG finite element methods we refer to \cite{Arnold2002, ern2012, hesthaven2007}. For a review on dG finite element methods in the context of wave equations see \cite{cohen2016finite}. Furthermore, we derive strong error estimates for the spatial discretization of problem (\ref{wavesde}). 

\subsection{Discontinuous Galerkin method and semigroup approach}
\label{sec:dgmethod}
In order to discretize problem (\ref{wave}) in space, we consider a family of regular and  quasi-uniform meshes  
${\cal T}_h$ parametrized by the mesh-width $h > 0$. Each ${\cal T}_h$ partitions the domain ${\cal D}$ into elements $\mathfrak{T}$, where we denote by $h_{\mathfrak{T}}$ the diameter of the element $\mathfrak T$ and $h={\rm max}_{\mathfrak T} h_{\mathfrak T}$ the mesh-width. We assume that the elements $\mathfrak T$ are triangles or parallelograms in two space dimensions, and tetrahedra or parallelepipeds in three dimensions, respectively. We denote by $\mathfrak{ F}_h = \cup_{\mathfrak T \in \mathcal{ T}_h} \partial \mathfrak T$ the set of all faces. This is split into  boundary $\mathfrak{ F}_h^B=\mathfrak{ F}_h\cap \partial {\cal D}$ and interior faces   $\mathfrak{ F}_h^I=\mathfrak{ F}_h\setminus\mathfrak{ F}_h^B$. Let $\mathfrak T^+$, $\mathfrak T^-$ be two elements sharing an interior face $\mathfrak{ F}\in \mathfrak{ F}_h^I$ with respective outward normal unit vectors $\bf{n}^+$ and $\bf{n}^-$.  Denoting by $u^{\pm}$ the trace of $u:{\cal D}\rightarrow \mathbb{R}$ taken from within $\mathfrak T^{\pm}$, we define the average of $u$ over $\mathfrak{ F} \in \mathfrak{ F}_h^I$ by 
$$ 
 \bm{ \{ } { u} \bm{ \} } = \frac{1}{2}\left({u}^+ + {u}^- \right).
$$
Similarly, the jump of ${ u}$ over $\mathfrak{ F} \in \mathfrak{ F}_h^I$ is given by 
$$ 
\llbracket {u} \rrbracket  ={ u}^+ {\bf n}^{+} + {u}^- \bf{n}^-.
$$
For a boundary face $\mathfrak{ F}\in \mathfrak{ F}_h^B$, we set $\bm{ \{ } { u} \bm{ \} }={u}$ and $\llbracket {u} \rrbracket = {u}\bf{n}$, where $\bf{n}$ denotes the unit outward normal vector on $\partial {\cal D}$. 
\par 
We now define the discontinuous polynomial space 
\begin{equation} 
V_h=\left\{ u\in L_2\left( {\cal D}\right):u|_{\mathfrak T}\in {\cal P}^p(\mathfrak T),\, \mathfrak T\in {\cal  T}_h \right\},
\end{equation} 
where ${\cal P}^p(\mathfrak T)$  denotes the polynomials of (total) degree less or equal to $p \geq 1$.

\par
The dG semi-discrete formulation of (\ref{wave}) is
given by: find $u_h(t)\in V_h$ such that 
\begin{align} \label{discdgwave}
&\left({\rm d} \dot{u}_h, v\right) +B_h\left( u_h,v\right){\rm d}t=\left( P_hf(u_h),v\right)dt+\left( P_h{\rm d}W,v\right) \quad \forall v\in V_h,\;t\in(0,T),\\
 &\,u_h(\cdot,0) =  u_{h,0}, \,\dot{u}_h(\cdot,0)=  v_{h,0},
\end{align}
where $(\cdot,\cdot)$ is the $L_2(\cal D)$ inner product, $P_h: L_2\left( {\cal D}\right)\rightarrow V_h$  the $L_2$-projection onto $V_h$,  $u_{h,0}, v_{h,0} \in V_h$ projections of initial data to be determined later, and $B_h$  the  symmetric interior penalty discrete bilinear form
\begin{equation}  \label{bilinear}
\begin{aligned}
B_h\left( u,v \right) =& \sum_{\mathfrak T\in \mathcal{T}_h}\int_{\mathfrak T}{\rm \bf D} \nabla u \cdot \nabla v\,{\rm d}x - \sum_{\mathfrak{F}\in\mathfrak{F}_h}\int_{\mathfrak{F}} \bm{ \{ } {\rm \bf D}\nabla u \bm{ \} }  \cdot\llbracket v\rrbracket \,{\rm d}s  \\
&- \sum_{\mathfrak{F}\in\mathfrak{F}_h}\int_{\mathfrak{F}}\llbracket u\rrbracket\cdot  \bm{ \{ } {\rm \bf D}\nabla v \bm{ \} }  \,{\rm d}s + \sigma_0  \sum_{\mathfrak{F}\in\mathfrak{F}_h}\int_{\mathfrak{F}}h_{\mathfrak{F}}^{-1}{\rm \bf D} \llbracket u \rrbracket\cdot  \llbracket v\rrbracket\, {\rm d}s,
\end{aligned}
\end{equation}
where $h_{\mathfrak{F}}$ is the diameter of the face $\mathfrak{F}$.  The interior penalty stabilization parameter $\sigma_0 > 0$ has to be chosen sufficiently large but independent of the mesh size. The last three terms in (\ref{bilinear}) correspond to jump and flux terms at the faces and they vanish when $u, v \in H^2({\cal D})\cap H_0^1({\cal D})$. The third term in (\ref{bilinear}) makes the bilinear form symmetric   and the last term ensures coercivity of the  bilinear form, see Lemma \ref{coer_and_cont}.

The bilinear form $B_h(\cdot,\cdot)$ defines a discrete linear operator $\Lambda_h : V_h \rightarrow V_h$
$$
(\Lambda_h v_h, w) = B_h(v_h,w), \qquad \forall w \in V_h.
$$
This in turn gives a discrete analogue of the norm $\| \cdot \|_\alpha$ 
$$
\|v_h\|_{h,\alpha} := \|\Lambda_h^{\alpha/2}v_h\|=\left( \sum\limits_{j=1}^{N_h}\lambda_{h,j}^\alpha(v_h,\phi_{h,j})^2\right)^{1/2}, \qquad v_h \in \dot{H}^\alpha_h,\, \alpha \in \mathbb{R},
$$
where $\{\phi_{h,j})\}_{j=1}^{N_h}$, $N_h=\dim V_h$, are the orthonormal eigenvectors of $\Lambda_h$ with corresponding eigenvalues $\lambda_{h,j}\geq 0$.
Note that since $\Lambda_h$ is a symmetric, positive definite operator, the fractional power is well-defined.
We also introduce discrete variants of $\vertiii{\cdot}_\alpha$ and ${\cal H}^\alpha$
\begin{equation*}
{\cal H}_h^\alpha=V_h\times V_h, \quad\vertiii{[u_1,u_2]^T}_{h,\alpha}^2 = \| u_1\|_{h,\alpha}^2 + \| u_2\|_{h, \alpha-1}^2,\quad [u_1,u_2]^T\in V_h\times V_h.
\end{equation*}
We now introduce the broken norm as in \cite{Grote2009}
\begin{equation} \label{eq:starnorm}
\|u\|_{*}: = \left( \sum_{\mathfrak T\in{\cal T}_h}\|\nabla u \|^2_{L_2(\mathfrak T)} + \sum_{\mathfrak T\in{\cal T}_h}h_{\mathfrak T}^2\| \Delta u \|^2_{L_2(\mathfrak T)}+\sum_{\mathfrak{F}\in{\mathfrak{ F}}_h}h_{\mathfrak{F}}^{-1}\| \llbracket u \rrbracket  \|^2_{L_2(\mathfrak{F})}\right)^{1/2}.
\end{equation}

The bilinear form $B_h$  in (\ref{bilinear}) is coercive and continuous in  the  norm (\ref{eq:starnorm}), see \cite{Arnold1982, Arnold2002}.
\begin{lemma} \label{coer_and_cont}
For large enough $\sigma_0>0$ there exists a constant $C_A>0$, dependent on $\sigma_0, d_{\max},$ and $d_{\min}$, and  independent of the mesh size, such that 
\begin{equation} \label{continuity}
\left| B_h(u,v)\right|\leq C_A\| u\|_{*} \| v\|_{*}, \quad \forall u, v \in {\dot H}^2\left( {\cal D}\right) + V_h,
\end{equation}
and 
\begin{equation} \label{coercivity}
B_h(u,u) \geq \frac{1}{2}\|u\|^2_{*}, \quad \forall u\in V_h.
\end{equation}
Consequently, we have the following norm equivalence
\begin{equation} \label{eq:energynormeq}
\frac{1}{2}\|u\|^2_{*} \leq \|u\|_{h,1}^2 \leq C_A \|u\|_*^2, \qquad \forall u \in V_h. 
\end{equation}
\end{lemma}
We also need the following spectral estimate (Lemma 3.3 in \cite{Grote2009}).
\begin{lemma} \label{lem:uniformes}
For $ u\in V_h+\dot{H}^2({\cal D})$, it holds 
\begin{equation} 
B_h(u,u) \leq C_sh^{-2}\|u\|^2, 
\end{equation}
where $C_s >0$,  is a constant independent of the mesh size,  and depends on $\sigma_0, d_{\max},$ and  the polynomial degree p.
\end{lemma}
\begin{rem}
 By Lemma \ref{lem:uniformes} we obtain the following bound for the eigenvalues of the discrete operator $\Lambda_{h}$
\begin{equation} \label{eq:eigenes}
\lambda_{h,j}\leq C_sh^{-2}, \quad j=1,\dots, N_h,
\end{equation}
since $(\Lambda_h\phi_{h,j},\phi_{h,j}) =B_h(\phi_{h,j},\phi_{h,j})=\lambda_{h,j}(\phi_{h,j},\phi_{h,j})$.
Additionally, we deduce the inverse estimate for any $u\in V_h$
\begin{equation} \label{eq:inverse}
\begin{aligned}
\|u\|_{h,\alpha} = \|\Lambda_h^{\alpha/2}u\| = \left(\sum\limits_{j=1}^{N_h} \lambda_{h,j}^{\alpha}(u,\phi_{h,j})^2\right)^{1/2}\leq  \sqrt{C_s}h^{-1}\|u\|_{h,\alpha-1}.
\end{aligned}
\end{equation}
\end{rem}

 The dG semi-discrete analogue of the first order formulation \eqref{wavesde} is: find $U_h=\left[ \ u_{h,1}, u_{h,2}
\right]^T\in V_h\times V_h$ such that  
\begin{equation} \label{eq:dgsde}
\begin{aligned}
&{\rm d}U_h(t) = A_h U_h(t){\rm d}t + F(U_h(t)){\rm d}t+ BP_h{\rm d}W,\quad t\in(0,T),\\
&U_h(\cdot,0)= U_{h, 0},
\end{aligned}
\end{equation}
where 
\begin{equation} \label{discreteA}
A_h = \begin{bmatrix} 0 &I \\ -\Lambda_h & 0 \end{bmatrix}, \quad U_{h, 0} = \left[ \begin{matrix}  u_{h, 0} \\ v_{h, 0}\end{matrix}\right], \quad F(U_h(t)) = \left[ \begin{matrix}  0 \\ u_{h,1}(t)
\end{matrix}\right],
\end{equation}
 and $B$ is as in \eqref{eq:defns_1storder}.
In order to ensure existence and uniqueness of problem (\ref{eq:dgsde}), we  first need to show that the discrete operator $A_h:V_h\times V_h\rightarrow V_h\times V_h$  satisfies the hypothesis of the Hille-Yosida Theorem \cite[Theorem 3.5]{engel2006}, i.e., generates  a strongly continuous contraction semigroup   on $V_h\times V_h$.
\begin{prop} \label{prop:hille}
The discrete operator $A_h:V_h\times V_h\rightarrow V_h\times V_h$ in (\ref{discreteA}) generates a strongly continuous contraction semigroup on $V_h\times V_h$.
\end{prop}
\begin{proof}
The proof of the proposition follows from \cite[Chapter 10.3]{brezis2010}.
\end{proof}
Similarly to (\ref{semigroup}), the $C_0$-semigroup $E_h(t)$ generated by the discrete operator $A_h$ is given by 
\begin{equation} \label{semigroupdiscrete}
E_h(t)=e^{tA_h}=\begin{bmatrix}
C_h(t) & \Lambda_h^{-1/2}S_h(t)\\
-\Lambda_h^{1/2}S_h(t) & C_h(t)
\end{bmatrix},
\end{equation}
where $C_h(t)=\cos\left(t\Lambda_h^{1/2}\right)$ and $S_h(t)=\sin\left(t\Lambda_h^{1/2}\right)$. 

Similarly to the continuous case, see Theorem~\ref{thm:cont_exist}, we have the existence of the mild solution to the semi-discrete system. 

\begin{lemma}\label{lem:disc_exist}
Assume that $W(t)$ satisfies Assumption \ref{assum:noise} and that  f satisfies (\ref{eq:assumf}). Also let $\|U_{h,0}\|_{L_2(\Omega,{\cal H}^{\beta}_h)}< \infty$, then  the dG formulation (\ref{eq:dgsde})  has a unique mild solution given by 
\begin{equation} \label{eq:discretemild}
U_h(t) = E_h(t)U_{h,0} + \int_0^tE_h(t-s)P_hF(U_h(s)){\rm d}s+ \int_0^tE_h(t-s)BP_h{\rm d}W(s).
\end{equation}
Further,   there exists a constant $C : = C(T, \|U_{h,0}\|_{L_2(\Omega,{\cal H}^{\beta}_h)}, \| \Lambda^{(\beta-1)/2}Q^{1/2}\|_{\rm HS})$, and  independent of $h$ such that 
\begin{equation} \label{eq:mildstab}
\| U_h(t)\|_{L_2(\Omega, {\cal H}^\beta_h)}\leq C, \quad 0\leq t\leq T.
\end{equation}
\end{lemma}
\begin{proof} 
The proof follows from \cite[Proposition 3]{anton2015} by using that $\| v\|_{h,\beta} =\|\Lambda_h^{\beta/2} v\|$.
\end{proof}

\subsection{Strong convergence in space}
\label{sec:errorspace}
In this subsection, we prove strong convergence of the dG approximation of the stochastic wave equation (\ref{eq:dgsde}), first with respect to the broken norm (\ref{eq:starnorm}) and then with respect to the $L_2$-norm.
 To analyze the strong convergence of the spatial approximation  (\ref{eq:dgsde}), we need to derive error estimates for the sine and cosine operators as in Corollary 4.2 in \cite{Kovacs2010}.  Before we state strong error estimates for the semi-discrete dG formulation (\ref{eq:dgsde}) we derive optimal error bounds for the deterministic homogeneous wave equation.

{\bf Deterministic homogeneous wave equation.}
We now look at the deterministic  homogeneous wave equation 
\begin{equation} \label{eq:homwave}
\begin{aligned}
&{\rm d}\dot{u}+\Lambda u  {\rm d}t =0 &&{\rm in}\;\; {\cal D}\times(0,T),\\
&\; u= 0  &&{\rm on } \; \;\partial {\cal D} \times(0,T),\\
&\;  u(\cdot,0) = u_0,\; \dot{u}(\cdot,0)= v_0 & &{\rm in} \;\; {\cal D}.
\end{aligned}
\end{equation}
The dG semi-discrete  formulation of (\ref{eq:homwave}) is: find $u_h(t)\in V_h$ such that
\begin{equation} 
\begin{aligned} \label{eq:dwave}
&\left( {\rm d} \dot{u}_h, v\right) +B_h\left( u_h,v\right){\rm d}t=0\quad \forall v\in V_h,\;t\in(0,T),\\
 &\,u_h(\cdot,0) =  u_{h,0}, \, \dot{u}_h(\cdot,0)=  v_{h,0}, \quad u_{h,0}, v_{h,0}\in V_h.\\
\end{aligned}
\end{equation}

We recall some useful results for the dG finite element method. 
For $ u \in \dot{H}^2({\cal D})$, the Galerkin projection $\Pi_h u\in V_h$ is defined as follows
\begin{equation} \label{eq:galerkpro}
B_h(\Pi_hu -u,v)=0, \quad v\in V_h.
\end{equation}

Since $\partial_t^i(\pi_Iu)=\pi_I(\partial_t^iu )$, $i=0,\dots,2$, where $\pi_I$ can be chosen as $P_h$ or $\Pi_h$, we have the following error bound in the $L_2$-norm 
\begin{equation}\label{eq:l2esproj}
\|\partial_t^i(u-\pi_Iu)\|\leq C h^{p+1}\| \partial_t^{i}u\|_{p+1},\qquad \partial_t^{i}u \in\dot{H}^{p+1}, \, p\geq1,
\end{equation}
where we recall that $p$ is the (local) polynomial degree of the discrete space $V_h$. 
The error estimate for the Galerkin projection in the broken norm (\ref{eq:starnorm})  is
\begin{equation}\label{eq:ellespro}
\| u-\Pi_hu\|_* \leq C h^p\|u\|_{p+1},\qquad u\in \dot{H}^{p+1}, \, p\geq1.
\end{equation}
Estimates (\ref{eq:l2esproj}) and (\ref{eq:ellespro}) can be found in Lemma 4.1 in \cite{Grote2009}.
\begin{theorem}
\label{the:deteres}
Let the exact solution $u$ of (\ref{eq:homwave}) satisfy  
\begin{equation*} 
u,\, \dot{u},\ddot{ u}\in L^\infty([0,T]; \dot{H}^{p+1}({\cal D})),
\end{equation*}
for $p\geq 1$, and  $u_h$ be the dG approximation obtained by (\ref{eq:dwave}). Setting $e(t)=u(t)-u_h(t)$, $t\in[0,T]$, we have for a constant  $C>0$, independent of the mesh size $h$,
\begin{align} 
\label{eq:enes}
&\| e(t)\|_* \leq C\left\{ \| u_0-\Pi_h u_0 \|_{*}+\|v_0 - \Pi_h v_0 \|\right\}
+h^{p} \left\{ \|u\|_{p+1} +\int_0^t \|\ddot{u}(s)\|_{p} \,{\rm d}s\right\},\\
\label{eq:l2vel}
&\|\dot{e}(t)\| \leq C\left\{ \| u_0-\Pi_h u_0 \|_{h,1}+\|v_0 - \Pi_h v_0 \|\right\}
+h^{p+1} \left\{ \| \dot{u}\|_{p+1} +\int_0^t \|\ddot{u}(s)\|_{p+1} \,{\rm d}s\right\},\\
\label{eq:l2dis}
\begin{split}
&\| e(t)\|\leq   C\left\{ \| u_0-\Pi_h u_0 \|+\|v_0 - P_h v_0 \|_{h,-1}\right\}\\
& \qquad \qquad \qquad + h^{p+1}\left\{
 \|u(s)\|_{p+1} + \int_0^t \|\dot{u}(s)\|_{p+1} \,{\rm d}s \right\}.
\end{split}
\end{align}
\end{theorem}
\begin{proof}
We set as in \cite{larsson2008}
\begin{equation} \label{eq:erreq}
e = u - \Pi_h u + \Pi_h u - u_h =\rho + \theta. 
\end{equation}
Then, using the Galerkin projection (\ref{eq:galerkpro}), the error satisfies 
\begin{equation} \label{eq:erreqe} 
(\ddot{\theta},v)  + B_h(\theta, v) = - (\ddot{\rho},v) \qquad \forall v\in V_h. 
\end{equation}
Choosing $v=\dot{\theta}$ and using (\ref{coercivity}), we conclude in the standard way that, see \cite[Theorem 13.1]{larsson2008},
\begin{equation} 
\begin{aligned} 
\|\dot{\theta}(t)\| + \|\theta(t)\|_{h,1} \leq C\left\{  \| \theta(0) \|_{h,1}+ \|\dot{\theta}(0) \|+\int_{0}^t\|\ddot{\rho}(s) \| {\rm d}s\right\}.
\end{aligned}
\end{equation} 
By the triangle inequality, the norm equivalence \eqref{eq:energynormeq}, and estimates \eqref{eq:l2esproj} for $\pi_I=\Pi_h$, and \eqref{eq:ellespro}, we conclude \eqref{eq:enes} and \eqref{eq:l2vel}. 

 Although, \cite{Kovacs2010} uses continuous polynomials, the proof of  estimate (\ref{eq:l2dis}), follows along the same lines as \cite[Theorem 4.1]{Kovacs2010} by rewriting the problem in a first-order formulation and using estimates \eqref{eq:l2esproj} for the Galerkin projection \eqref{eq:galerkpro}.  
\end{proof}
\begin{rem} 
In \cite{grote2006} optimal convergence rates are derived for the displacement with  respect to the $L_2$-norm and with respect to the broken norm (\ref{eq:starnorm}) for \eqref{eq:homwave}. A  bound for the velocity in the $L_2$-norm follows from \cite[Theorem 4.1]{grote2006}, but is not optimal. Theorem \ref{the:deteres} provides optimal error estimates for both the displacement and the velocity with respect to the $L_2$-norm.
\end{rem} 
In the following lemma we state error estimates for the sine and cosine operators.
\begin{lemma} 
Denote $U_0=[u_0, v_0]^T$ and let 
\begin{equation} \label{FhGh}
\begin{aligned}
&G_h(t)U_0 =\left( C_h(t)\Pi_h - C(t)\right)u_0 +\Lambda_h^{-1/2}S_h(t)v_{h,0}-\Lambda^{-1/2}S(t)v_0,\\
&\dot{G}_h(t)U_0 = -\left(\Lambda_h^{1/2}S_h(t)\Pi_h-\Lambda^{1/2}S(t) \right)u_0+\left( C_h(t)\Pi_h-C(t)\right)v_0.
\end{aligned}
\end{equation}
Then there exists a constant  $C>0$  independent of the mesh size $h$ such that: 
\begin{enumerate}[i.]
\item  If $v_{h,0} = \Pi_h v_0$, then
\begin{equation} \label{Gherror}
\| G_h(t)U_0 \|_*\leq C(1+t)h^{\frac{p}{p+1}(\beta-1)}\vertiii{ U_0}_\beta,\quad t\in[0,T], \quad \beta \in[1,p+2].
\end{equation}
\item If $v_{h,0} = P_h v_0$, then 
\begin{equation} \label{Fherror}
\| G_h(t)U_0 \|\leq C(1+t)h^{\frac{p+1}{p+2}\beta}\vertiii{ U_0}_\beta, \quad t\in[0,T],\quad \beta \in[0,p+2].
\end{equation}
\item [iii.] If $v_{h,0} = \Pi_h v_0$, then
\begin{equation}\label{dotGherror}
\| \dot{G}_h(t)U_0 \|\leq C(1+t)h^{\frac{p+1}{p+2}(\beta-1)}\vertiii{ U_0}_\beta, \quad t\in[0,T], \quad \beta \in[1,p+3].
\end{equation}
\end{enumerate}
\end{lemma}
\begin{proof}
In \cite{Kovacs2010},  problem (\ref{eq:dwave}) is discretized in space by using piecewise continuous polynomials and error bounds are derived for the $G_h$ and $\dot{G}_h$ operators in terms of initial conditions in Corollary 4.2.
The proof of the above lemma follows from \cite{Kovacs2010},
 employing the estimates  in Theorem \ref{the:deteres} and by the norm equivalence \eqref{eq:energynormeq}.
\end{proof}

As a corollary, we obtain  the  following error estimates.
\begin{theorem} 
\label{the:dgmethodes}
Assume that  f satisfies \eqref{eq:assumf}. Denote $U_0=[u_0, v_0]^T$,  and let $U=[u_1, u_2]^T$ and   $U_h=[u_{h,1},u_{h,2}]^T$ be given by (\ref{mildsolution}) and (\ref{eq:discretemild}), respectively. Choosing $u_{h, 0}=\Pi_h u_0$, we have, for $t\in[ 0,T]$  the following estimates: 
\begin{enumerate}[i.]
\item   If $v_{h,0} =P_h v_0$ and $W(t)$ satisfies Assumption \ref{assum:noise},  for some  $\beta \geq  0$, then 
\begin{equation} \label{stoerror1}
\| u_{h,1}(t)-u_1(t)\|_{L_2\left( \Omega,L_2({\cal D})\right)}\leq Ch^{\min(\frac{p+1}{p+2}\beta,p+1)}.
\end{equation}
\item  If $v_{h,0} = \Pi_h v_0$ and  $W(t)$ satisfies Assumption \ref{assum:noise} for some $\beta \geq 1$, then 
\begin{equation} \label{stoerrordgnorm}
\| u_{h,1}(t)-u_1(t)\|_{L_2\left( \Omega, \dot{H}_h^1\right)}\leq C h^{\min(\frac{p}{p+1}(\beta-1),p+1)}.
\end{equation}
\item  If $v_{h,0} = \Pi_h v_0$ and $W(t)$ satisfies Assumption \ref{assum:noise}  for some $\beta \geq 1$, then 
\begin{equation} \label{stoerror2}
\| u_{h,2}(t)-u_2(t)\|_{L_2\left( \Omega,L_2({\cal D})\right)}\leq  Ch^{\min(\frac{p+1}{p+2}(\beta-1),p+1)}.
\end{equation}
\end{enumerate}
The constant $C$ depends on $t, \|U_{0}\|_{L_2(\Omega,{\cal H}^{\beta})}, \|\Lambda^{(\beta-1)/2}Q^{1/2}\|_{\rm HS}$, and is independent of $h$.
\end{theorem}
\begin{proof} 
 Estimates (\ref{stoerror1}) and (\ref{stoerror2}) are proved as in \cite[Theorem 4]{anton2015} where piecewise continuous polynomials are used for the spatial discretization.
The proof of estimate \eqref{stoerrordgnorm} follows along the same lines with the proof of  (\ref{stoerror1}).

\end{proof}

\section{The stochastic position Verlet method}
\label{sec:fullydiscrete}
We now  consider the full discretization of the  stochastic wave equation (\ref{wave}). Our goal  is to prove optimal convergence of the strong error for the full discretization of problem (\ref{wave}).  
 Let $\tau >0$ be the time step size, so that $t_n = n\tau$, $n= 1, 2, \dots, N$, and $T=N\tau$. Letting $X^n=[X_1^n, X_2^n]^T$ be the numerical approximation of $U_h(t_n)$ in (\ref{eq:dgsde}), the stochastic position Verlet (SVM) scheme is 
\begin{equation} \label{eq:stochasticsv}
\begin{aligned}
&X_1^{n-1/2} = X_1^{n-1} + \frac{\tau}{2} X_2^{n-1},\\
&X_2^{n} = X_2^{n-1} -\tau\Lambda_h X^{n-1/2}_1 + \tau P_h f(X_1^{n-1/2})+  P_h\Delta W^n,\\
&X_1^{n} = X^{n-1/2}_1 + \frac{\tau}{2} X_2^{n},
\end{aligned}
\end{equation}
where $X^0 = U_{h,0}$ and $\Delta W^n = W(t_{n}) - W(t_{n-1})$. 

\subsection{ Stability of the scheme}
To study the stability of the above scheme, we rewrite system (\ref{eq:stochasticsv}) as follows 
\begin{equation} \label{eq:stochasticsv_1storder}
X^{n} = M X^{n-1} +\tau DP_hf(X^{n-1/2}_1) + D P_h\Delta W^n, 
\end{equation} 
where $D = \left[ \frac{\tau}{2}I, I \right]^T$, and
\begin{equation} \label{eq:matrixM}
M(\tau)=\left[ \begin{matrix} I-\frac{\tau^2}{2}\Lambda_h & \tau I - \frac{\tau^3}{4}\Lambda_h  \\ -\tau \Lambda_h& I - \frac{\tau^2}{2}\Lambda_h\end{matrix}\right].
\end{equation}
 By recursion, the approximate solution $X^{n}$ can be written as
\begin{equation} \label{eq:svnsteps2}
X^{n} = M^{n}X^0 + \sum\limits_{j=1}^{n}M^{n-j}\tau DP_hf(X^{j-1/2}_1)+ \sum\limits_{j=1}^{n}M^{n-j}DP_h\Delta W^{j}.
\end{equation}

We study the stability of the scheme under a CFL  condition \cite{CFL}. More specifically, we assume that the mesh size $h$ and the time step $\tau$ satisfy the following restriction
\begin{equation} \label{eq:CFL}
\tau <C_{\rm CFL}h,
\end{equation}
where $C_{\rm CFL}<\frac{2}{\sqrt{C_S}}$  and $C_s$ is the constant from Lemma \ref{lem:uniformes}.
\subsubsection{Discrete norm}
Under the  CFL condition \eqref{eq:CFL}, we  introduce the following inner product for  $v=[v_1, v_2]^T$,  $w=[w_1,w_2]^T \in {\cal H}_h^{\alpha}$
\begin{equation} \label{eq:innerm}
 \langle v,w\rangle_{m,\alpha} =(V^{-1}w)^* \begin{bmatrix}
\Lambda_h^{\alpha-1} & 0 \\
0 & \Lambda_h^{\alpha-1}  
\end{bmatrix} V^{-1}v, \quad \alpha\in\mathbb{R},
\end{equation}  
where $(V^{-1}w)^*$ is the conjugate transpose, and
\begin{equation*}
V^{-1}=\left[ \begin{matrix}i \Lambda_h^{1/2}(I-\frac{\tau^2}{4}\Lambda_h)^{-1/2} /2& I/2\\  -i\Lambda_h^{1/2}(I-\frac{\tau^2}{4}\Lambda_h)^{-1/2}/2 & I/2\end{matrix}\right].
\end{equation*} 
The inner product $ \langle \cdot,\cdot\rangle_{m,\alpha}$ defines the discrete norm 
\begin{equation} \label{eq:discretemnorm}
\|v\|_{m, \alpha}^2 = \langle v,v\rangle_{m,\alpha} \quad \alpha\in \mathbb{R},\, v=[v_1, v_2]^T\in {\cal H}_h^{\alpha}.
\end{equation} 

A key result for the stability analysis of our scheme is showing that the time integrator $M$ \eqref{eq:matrixM} preserves the $\| \cdot \|_{m,\alpha}$ norm. To do so, we consider the  following spectral decomposition of the matrix $M$  in \eqref{eq:matrixM}. 
\begin{lemma}
\label{lem:spectraldec}
  Under the CFL condition \eqref{eq:CFL} we have that
\begin{equation*}
M = VDV^{-1}
\end{equation*}
where  $V=
\begin{bmatrix}
  m_+ & m_-
\end{bmatrix}$ and 
 $D=\operatorname{diag}(\mu_+,\mu_-)$ with
\begin{equation} \label{eq:eigenM}
\mu_{\pm}=I-\frac{\tau^2\Lambda_{h}}{2}\pm i\tau\Lambda_{h}^{1/2}(I-\tau^2\Lambda_{h}/4)^{1/2},
\end{equation}
and
\begin{equation} \label{eq:vecM}
m_{\pm} =\left[ \begin{matrix}\pm i\Lambda_h^{-1/2}(I-\frac{\tau^2}{4}\Lambda_{h})^{1/2}\\I \end{matrix}\right].
\end{equation}
Further, $ \mu_+^* \mu_+=I$ and $\mu_-^* \mu_-= I$, where $\mu_\pm^*$ is the adjoint  with respect to the $L^2$-inner product, respectively .
\end{lemma}
\begin{proof}
  The expressions \eqref{eq:eigenM}  and \eqref{eq:vecM} can be verified by direct computation. We first note that $\mu_+^*$ is given by
\begin{equation*} 
\mu_+^* = I-\frac{\tau^2\Lambda_{h}}{2}- i\tau\Lambda_{h}^{1/2}(I-\tau^2\Lambda_{h}/4)^{1/2} = \mu_-,
\end{equation*}
since $\Lambda_h$ is a real symmetric operator.  Similarly,   $\mu_-^* = \mu_+$. Then it is not difficult to see that $\mu^*_+\mu_+= \mu_-\mu_+= I$ and $\mu_-^*\mu_- = \mu_+ \mu_- = I$.  
\end{proof}

\begin{lemma}
\label{lem:boundm}
Let $v=[v_1,v_2]^T\in {\cal H}_h^\alpha$, under the CFL condition \eqref{eq:CFL},
\begin{equation} 
 \label{eq:boundm}
\|Mv\|_{m,\alpha} = \|v\|_{m,\alpha}, \quad \alpha\in\mathbb{R}.
\end{equation}
\end{lemma}
\begin{proof}
Using the spectral decomposition of M in Lemma \ref{lem:spectraldec} and the definition of the $\|\cdot\|_{m,\alpha}$ norm in \eqref{eq:discretemnorm}, we have  $v=[v_1,v_2]^T\in {\cal H}_h^\alpha$
\begin{equation*}
\begin{split}
\| Mv\|_{m,\alpha}^2 &= \| VDV^{-1}v\|_{m,\alpha}^2 = (V^{-1}VDV^{-1}v)^*\begin{bmatrix}
\Lambda_h^{\alpha-1} & 0 \\
0 & \Lambda_h^{\alpha-1}  
\end{bmatrix} V^{-1}VDV^{-1}v\\
& =  (DV^{-1}v)^* \begin{bmatrix}
\Lambda_h^{\alpha-1} & 0 \\
0 & \Lambda_h^{\alpha-1}  
\end{bmatrix}DV^{-1}v
= (V^{-1}v)^*D^* \begin{bmatrix}
\Lambda_h^{\alpha-1} & 0 \\
0 & \Lambda_h^{\alpha-1}  
\end{bmatrix}DV^{-1}v\\
&= v^*(V^{-1})^*\begin{bmatrix}
\Lambda_h^{\alpha-1} & 0 \\
0 & \Lambda_h^{\alpha-1}  
\end{bmatrix}V^{-1}v=\|v\|_{m,\alpha}^2,
\end{split}
\end{equation*}
since $\mu_{\pm}$ and $\Lambda_h^{\alpha-1}$ commute; recall $D={\rm diag}(\mu_+,\mu_-)$, and  

\begin{equation*}
\begin{split}
D^*D & = \begin{bmatrix}
\mu_+^* & 0 \\
0 &\mu_-^*  
\end{bmatrix} \begin{bmatrix}
\mu_+ & 0 \\
0 &\mu_-
\end{bmatrix}= \begin{bmatrix}
\mu_+^*\mu_+& 0 \\
0 &\mu_-^*\mu_-
\end{bmatrix} 
=  \begin{bmatrix}
I & 0 \\
0 &I
\end{bmatrix}.
\end{split}
\end{equation*}
\end{proof}

Next, we prove that the  norms $\vertiii{\cdot}_{h, \alpha}$ and $\|\cdot\|_{m,\alpha}$ are equivalent.
\begin{lemma} 
\label{lem:normeq}
Under the CFL condition (\ref{eq:CFL}), there exist constants $C_1, C_2>0$, independent of the mesh size $h$ and the time step $\tau$, such that 
\begin{equation} \label{eq:normeq} 
C_1\vertiii{v}_{h, \alpha}^2 \leq \|v\|_{m ,\alpha}^2\leq C_2 \vertiii{v}_{h, \alpha}^2\quad \alpha \in\mathbb{R},\, v\in {\cal H}_h^\alpha.
\end{equation} 
\end{lemma}
\begin{proof} 
By the definition of the norm  $\| \cdot\|_{m,\alpha}$  in (\ref{eq:discretemnorm}) we obtain
\begin{equation*}
\|v\|_{m, \alpha}^2
= \frac{1}{2}\|\Lambda_h^{\alpha/2}(I-\frac{\tau^2}{4}\Lambda_h)^{-1/2}v_1 \|^2 + \frac{1}{2}\|\Lambda_h^{(\alpha-1)/2}v_2 \|^2
\end{equation*}
Thus, we get under the CFL condition (\ref{eq:CFL})
\begin{equation} \label{eq:Mnorm}
\begin{aligned}
\|v\|_{m, \alpha}^2
&\leq \frac{1}{2} \| (I-\frac{\tau^2}{4}\Lambda_h)^{-1/2}\|^2_{{\cal L}(V_h)} \|\Lambda_h^{\alpha/2}v_1 \|^2 +\frac{1}{2}\|\Lambda_h^{(\alpha-1)/2}v_2 \|^2 \\
&\leq \max\{ C/2,1/2\} \vertiii{v}_{h,\alpha}^2= C_2 \vertiii{v}_{h,\alpha}^2.
\end{aligned}
\end{equation}
We also have under the CFL condition (\ref{eq:CFL})
\begin{equation} \label{eq:sobolevnorm}
\begin{aligned}
\vertiii{v}^2_{h, \alpha} &= \| \Lambda_h^{\alpha/2}v_1\|^2 + \| \Lambda_h^{(\alpha-1)/2}v_2\|^2\\
& \leq \|( I-\frac{\tau^2}{4}\Lambda_h)^{1/2}\| ^2_{{\cal L}(V_h})\| \Lambda_h^{\alpha/2}(I-\frac{\tau^2}{4}\Lambda_h)^{-1/2}v_1\|^2 + \| \Lambda_h^{(\alpha-1)/2}v_2\|^2\\
&\leq C \| \Lambda_h^{\alpha/2}(I-\frac{\tau^2}{4}\Lambda_h)^{-1/2}v_1\|^2 + \| \Lambda_h^{(\alpha-1)/2}v_2\|^2 \leq C_1 \|v\|_{m,\alpha}^2.
\end{aligned}
\end{equation}
Estimates (\ref{eq:Mnorm})  and  (\ref{eq:sobolevnorm}) complete the proof of (\ref{eq:normeq}).
\end{proof}

We also want to analyse the stability SVM applied to the linear analogue of \eqref{wave}. The linear stochastic wave equation is given by 
\begin{equation} \label{eq:linearsto}
{\rm d}\dot{u} = -\Lambda u {\rm d}t + {\rm d}W,
\end{equation}
with initial conditions as in \eqref{wave}, i.e., $u(\cdot,0)=u_0, \,\dot{u}(\cdot,0)= v_0$. Following notation from Section \ref{sec:dgmethod}, the dG approximation to the linear stochastic wave equation is given with \eqref{eq:dgsde} for $f\equiv 0$, i.e.,  find $U_h(t)=[u_{h,1}, u_{h,2}]^T\in V_h \times V_h$ such that 
\begin{equation*} 
\begin{aligned}
&{\rm d}U_h(t) = A_h U_h(t){\rm d}t + BP_hdW,\quad t\in(0,T),\\
&U_h(\cdot,0)= U_{h, 0}.
\end{aligned}
\end{equation*}
Let $Y^n=[Y_1^n, Y_2^n]^T$ be the temporal approximation to  the above problem. Then $Y^n$ is given by \eqref{eq:svnsteps2} for $f=0$, i.e.,
\begin{equation} \label{eq:linearprob}
Y^{n} = M^{n}Y^0 + \sum\limits_{j=1}^{n}M^{n-j}DP_h\Delta W^{j},
\end{equation}
where $Y^0=X^0=U_{h,0}$.

\begin{rem} 
Throughout this section, the constant $C$ denotes a generic positive constant that may vary from line to line and is  independent of $h$ and  $\tau$.
\end{rem}

In order to  prove the stability of the temporal approximation of the linear and nonlinear problem, we need  the following relation between $\Lambda_h$ and $\Lambda$ for a constant $C$, see proof of Theorem 4.4 in \cite{kovacs2012},
\begin{equation} \label{eq:lambdahlambda}
\| \Lambda_h^{-\delta} P_h \Lambda^\delta\|_{{\cal L}(L_2({\cal D}))} \leq C,\quad \delta \in [0,1/2].
\end{equation}
\begin{lemma} 
\label{lem:stabilitystoch}
Assume that $\| Y^0 \|_{L_2({\Omega, {\cal H}_h^\beta})}<\infty$ and that $\| X^0 \|_{L_2({\Omega, {\cal H}_h^\beta})}<\infty$. Let $W(t)$ satisfy Assumption \ref{assum:noise} for some $\beta\geq 0$ and let $f$ satisfy \eqref{eq:assumf}, then, under the CFL condition \eqref{eq:CFL}, there exists a constant $C$, independent of $h$ and  $\tau$, such that 
\begin{equation} \label{eq:stablinear}
\| Y^{n} \|_{L_2({\Omega, {\cal H}_h^\beta})}\leq C\left(\| Y^0 \|_{L_2({\Omega, {\cal H}_h^\beta})}+ t_n^{1/2}\|\Lambda^{(\beta-1)/2}Q^{1/2}\|_{\rm HS}\right)
\end{equation}
and
\begin{equation} \label{eq:stochsvstability}
\| X^{n} \|_{L_2({\Omega, {\cal H}_h^\beta})}\leq C\exp(Ct_n)\left(\| X^0 \|_{L_2({\Omega, {\cal H}_h^\beta})}+ t_n^{1/2} \|\Lambda^{(\beta-1)/2}Q^{1/2}\|_{\rm HS} + t_n\right).
\end{equation}
\end{lemma} 
\begin{proof} 
 Writing the increments $\Delta W^n$ as $\Delta W^n=\int_{t_{n-1}}^{t_{n}}{\rm d}W(s)$, we have for \eqref{eq:linearprob} by using  It\^{o}'s isometry \eqref{eq:itosisometry} and Lemma \ref{lem:boundm}
\begin{equation*}
\begin{split}    
\expect{\norm{Y^n}^2} &= \expect{ \norm{M^n Y^0}^2}+ \expect{\norm{\sum_{j = 1}^n\int_{t_{j-1}}^{t_{j}} M^{n-j} D P_h dW(s)}^2} \\
&= \expect{ \norm{Y^0}^2}+ \sum_{j = 1}^n(t_{j}-t_{j-1})\sum_{k = 0}^\infty\norm{D P_hQ^{1/2}e_k}^2\\
&= \expect{ \norm{Y^0}^2}+ t_n\sum_{k = 0}^\infty\norm{D P_hQ^{1/2}e_k}^2.
\end{split}
\end{equation*}
By the norm equivalence \eqref{eq:normeq}, it follows that
\begin{equation*}
\begin{split}    
\expect{\vertiii{Y^n}_{h,\beta}^2} \leq \frac{C_2}{C_1}\left( \expect{ \vertiii{X^0}_{h,0}^2}+ t_n\sum_{k = 0}^\infty\vertiii{D P_hQ^{1/2}e_k}_{h,\beta}^2\right).
\end{split}
\end{equation*}
Letting  $I=\sum_{k = 0}^\infty\vertiii{D P_hQ^{1/2}e_k}_{h,\beta}^2 $, employing the definition of the $\vertiii{\cdot}_{h,\beta}$ norm, and by the inverse estimate (\ref{eq:inverse}) we obtain 
\begin{equation*} 
\begin{aligned}
I\leq \frac{\tau^2C_s}{4h^2}\| \Lambda_h^{(\beta-1)/2}P_hQ^{1/2}\|_{\rm HS}^2+\| \Lambda_h^{(\beta-1)/2}P_hQ^{1/2}\|_{\rm HS}^2.
\end{aligned} 
\end{equation*}
Using the CFL condition (\ref{eq:CFL}) and \eqref{eq:lambdahlambda} for $\beta/2\in[-1,2]$, we finally get 
\begin{equation*}
 I\leq 2\| \Lambda^{(\beta-1)/2}Q^{1/2}\|_{\rm HS}^2.
\end{equation*}
The above estimate completes the proof of bound \eqref{eq:stablinear}.

 To prove estimate \eqref{eq:stochsvstability}, we use \eqref{eq:svnsteps2} and \eqref{eq:boundm} to get
\begin{equation*} 
\begin{aligned}
\mathbb{E}[ \|X^{n}\|_{m,\beta}^2]\leq &\,  3 \mathbb{E}[ \| X^0\|_{m,\beta}^2]  + 3\expect{\norm{ \sum\limits_{j=1}^n\int_{t_{j-1}}^{t_j}D P_h dW(s)}^2}\\
&\,+ 3n\tau^2\sum\limits_{j=1}^{n}\mathbb{E}\left[ \| DP_hf(X^{j-1/2}_1) \|_{m,\beta}^2\right].
\end{aligned}
\end{equation*}
Using \eqref{eq:normeq} and noting that the second term is bounded as above, we obtain
\begin{equation} \label{eq:xneq}
\begin{split}
\mathbb{E}\big[ \vertiii{X^{n}}_{h,\beta}^2\big] \leq C & \,\left( \mathbb{E}\left[ \vertiii{X^0}_{h,\beta}^2\right] +t_n\| \Lambda^{(\beta-1)/2}Q^{1/2}\|_{\rm HS}^2\right.\\
&\left. \quad+ n\tau^2 \sum\limits_{j=1}^{n} \mathbb{E}\left[ \vertiii{  DP_hf(X^{j-1/2}_1)}_{h,\beta}^2\right]\right).
\end{split}
\end{equation} 

Letting $II=n\tau ^2\sum\limits_{j=1}^{n} \mathbb{E}\left[ \vertiii{  DP_hf(X^{j-1/2}_1)}_{h,\beta}^2\right]$, we have by the definition of the norm $\vertiii{\cdot}_{h,\beta}$ and by  noting that  $\|\Lambda_h^{-1/2}\|_{{\cal L}(L_2({\cal D}))}^2\leq C$ , 
\begin{equation*} 
\begin{split}
II &= n\tau^2  \sum_{j=1}^{n}\mathbb{E}\left[\frac{\tau^2}{4}\| P_hf(X^{j-1/2}_1)\|_{h,\beta}^2+ \| P_hf(X^{j-1/2}_1)\|_{h,\beta-1}^2 \right]\\
&\leq(C+\frac{\tau^2}{4}) t_n \tau  \sum_{j=1}^{n} \mathbb{E}\left[\| P_hf(X^{j-1/2}_1)\|_{h,\beta}^2 \right]. 
\end{split}
\end{equation*}
Using (\ref{eq:assumf}) and  triangle inequality, we obtain 
\begin{equation*} 
\begin{split}
I I&\leq (C+\frac{\tau^2}{4}) t_n\tau  \sum_{j=1}^{n}\mathbb{E}\left[ 1 + \| X^{j-1/2}_1\|_{h,\beta}^2 \right] \\
&\leq(C+\frac{\tau^2}{4})t_n\tau  \sum_{j=1}^{n} \left( 1+2\mathbb{E}\left[\| X^{j-1}_1\|_{h,\beta}^2+ \frac{\tau^2}{4}\|X_2^{j-1}\|_{h,\beta}^2 \right]\right).
\end{split}
\end{equation*}
By the inverse estimate (\ref{eq:inverse})  and  the CFL condition (\ref{eq:CFL}), we deduce for $II$
\begin{equation*}
\begin{split}
II&\leq  Ct_n\tau  \sum_{j=1}^{n}\left( 1+ 2\mathbb{E}\left[\| X^{j-1}_1\|_{h,\beta}^2+ \frac{\tau^2 C_s}{4h^2}\|X_2^{j-1}\|_{h,\beta -1}^2 \right]\right)\\
&\leq Ct_n\tau  \sum_{j=1}^{n}\left( 1+2 \mathbb{E}\left[\| X^{j-1}_1\|_{h,\beta}^2+ \|X_2^{j-1}\|_{h,\beta-1 }^2 \right]\right)\\ 
& = Ct_n\tau  \sum_{j=1}^{n} \left( 1 +2\mathbb{E}\left[\vertiii{ X^{j-1}}_{h,\beta}^2 \right]\right).
\end{split}
\end{equation*}

Using the above estimate, we obtain for (\ref{eq:xneq})
\begin{equation*} 
\begin{split}
\mathbb{E}\big[ \vertiii{X^{n}}_{h,\beta}^2\big] \leq C&\left(  \mathbb{E}\big[ \vertiii{X^0}_{h,\beta}^2\big]+ t_n\|\Lambda^{(\beta-1)/2}Q^{1/2}\|_{\rm HS}^2\right.\\
& \left. \quad + t_n^2+ 2t_n\tau  \sum_{j=1}^{n} \left(\mathbb{E}\left[\vertiii{ X^{j-1}}_{h,\beta}^2 \right]\right)\right).
\end{split}
\end{equation*}
The discrete version of Gronwall’s inequality applied to the above inequality and taking square roots, gives
\begin{equation*} 
\begin{split}
\| X^n \|_{L_2({\Omega, {\cal H}_h^\beta})}\leq C\exp(Ct_n)\left(  \| X^0 \|_{L_2({\Omega, {\cal H}_h^\beta})}+ t_n^{1/2}\|\Lambda^{(\beta-1)/2}Q^{1/2}\|_{\rm HS} + t_n\right).
\end{split}
\end{equation*}
The above bound completes the proof of \eqref{eq:stochsvstability}.
\end{proof}

\subsection{Strong convergence analysis}
In this subsection, we derive strong error estimates for the full discretization of \eqref{wave}. 
Before we analyze the strong convergence of the temporal discretization (\ref{eq:svnsteps2}), we present  H\"{o}lder continuity of the  semi-discrete mild solution (\ref{eq:discretemild}). 
\begin{lemma} 
Let $U_h=[u_{h,1}, u_{h,2}]^T$ be the solution to the dG semi-discrete formulation (\ref{discdgwave}) given by (\ref{eq:discretemild}). Also let all conditions in Lemma \ref{lem:disc_exist} be fulfilled, then 
\begin{equation} \label{eq:holdersol}
\begin{split}
\mathbb{E}[ \| u_{h,1}(t)-u_{h,1}(s)\|_{h,0}^2] \leq C |t-s|^{2\min(\beta,1)},
\end{split}
\end{equation}
where the constant $C>0$ depends on $T, \| U_{h,0} \|_{L_2({\Omega, {\cal H}_h^\beta})}$,  and $\|\Lambda^{(\beta-1)/2}Q^{1/2}\|_{\rm HS}$, and is independent of $h$ and $\tau$.
\end{lemma}
\begin{proof}
  In \cite[Propositon 3]{anton2015}  H\"{o}lder continuity of the finite element approximation to the stochastic wave equation \eqref{wave} is proved. 
The proof for the dG approximation follows in the same way using that the discrete operator $\Lambda_h$ defines the discrete norm $\|v\|_{h,\beta}=\|\Lambda_h^{\beta/2}v\|$.
\end{proof}

We also need to  derive a bound for the difference between the semigroup $E_h$ (\ref{semigroupdiscrete}) and our time integrator M (\ref{eq:matrixM}). 

\begin{lemma} 
\label{lem:EminusM}
Let  $\alpha=0,1$.  The following estimate holds for the error between the  semigroup $E_h$ (\ref{semigroupdiscrete}) and the time integrator M (\ref{eq:matrixM}):
\begin{equation}\label{eq:EminusM}
\vertiii{(E_h-M)V}_{h,\alpha}\leq \tau^{3} \vertiii{V}_{h, \alpha+3},\qquad V=[v_1, v_2]^T\in {\cal H}_h^{\alpha+3}.
\end{equation}
\end{lemma} 
\begin{proof} 
We first note that
\begin{equation*} 
\begin{aligned}
(E_h - M)V=& \left[ \begin{matrix}(\cos(\tau\Lambda_h^{1/2})- (I-\frac{\tau^2\Lambda_h}{2}))v_1 + ( \Lambda_{h}^{-1/2}\sin(\tau\Lambda_h^{1/2})-(\tau I - \frac{\tau^3\Lambda_h}{4} ) )v_2\\ (-\Lambda_h^{1/2}\sin(\tau\Lambda_h^{1/2}) +\tau\Lambda_h )v_1+(\cos(\tau\Lambda_h^{1/2})-(I-\frac{\tau^2\Lambda_h}{2})) v_2 \end{matrix}\right].\\
\end{aligned}
\end{equation*}
Therefore, we have by the triangle inequality
\begin{equation} \label{eq:firstterm}
\begin{aligned}
\vertiii{ (E_h - M)V}_{h,\alpha}
\leq &\| (\cos(\tau\Lambda_h^{1/2})- (I-\frac{\tau^2\Lambda_h}{2}))v_1\|_{h,\alpha} \\
& +\|( \Lambda_{h}^{-1/2}\sin(\tau\Lambda_h^{1/2})-(\tau I - \frac{\tau^3\Lambda_h}{4} ) )v_2 \|_{h,\alpha} \\
&  + \| (-\Lambda_h^{1/2}\sin(\tau\Lambda_h^{1/2}) +\tau\Lambda_h )v_1 \|_{h,\alpha-1} \\
&+ \| (\cos(\tau\Lambda_h^{1/2})- (I-\frac{\tau^2\Lambda_h}{2}))v_2\|_{h,\alpha-1}\\
=& I+ II+ III+ IV.
\end{aligned}
\end{equation}
By the definition of the $\|\cdot \|_{h,\alpha}$ norm, we have for $I$
\begin{equation*}
\begin{split}
I=\left(\sum\limits_{j=1}^{N_h}\lambda_{h,j}^\alpha | \cos(\tau\sqrt{\lambda_{h,j}})-1 +\frac{\tau^2\lambda_{h,j}}{2}|^2(v_1,\phi_{h,j})^2 \right)^{1/2}.
\end{split}
\end{equation*}
Using Taylor's  theorem it holds that $|\cos(\tau\sqrt{\lambda_{h,j}}) -(1-\frac{\tau^2}{2}\lambda_{h,j})|=\frac{|\sin(\xi)|\tau^3\lambda_{h,j}^{3/2}}{6}\leq \frac{\tau^3\lambda_{h,j}^{3/2}}{6}$, for some $\xi\in(0,\tau\sqrt{\lambda_{h,j})}$, thus we get 
\begin{equation} \label{eq:firstterm2}
\begin{aligned}
I\leq \left(\sum\limits_{j=1}^{N_h} \lambda_{h,j}^\alpha| \frac{\tau^3\lambda_{h,j}^{3/2}}{6}|^2(v_1,\phi_{h,j})^2 \right)^{1/2}=\frac{\tau^3}{6}\|v_1 \|_{h,\alpha+3}.
\end{aligned}
\end{equation}
We now look at $II$ and again by  Taylor's theorem it holds that $| \lambda_{h,j}^{-1/2}\sin(\tau\sqrt{\lambda_{h,j}})-\tau|\leq \frac{\tau^3\lambda_{h,j}^{2}}{6}$, hence we  obtain
\begin{equation} \label{eq:secondterm3}
\begin{aligned}
II\leq  \|  (\Lambda_{h}^{-1/2}\sin(\tau\Lambda_h^{1/2})-\tau I)v_2 \|_{h,\alpha} + \frac{\tau^3}{4}\| v_2 \|_{h,\alpha+2} \leq \frac{5}{12}\tau^3\| v_2\|_{h,\alpha+2}.
\end{aligned}
\end{equation}
Similarly, we get  for $III$
\begin{equation} \label{eq:thirdterm}
\begin{aligned}
III\leq \left(\sum\limits_{j=1}^{N_h}\lambda^{\alpha-1}_{h,j} (\tau^3\lambda_{h,j}^2)^2(v_1,\phi_{h,j})^2\right)^{1/2}=\frac{\tau^3}{6}\| v_1\|_{h,\alpha+3}.
\end{aligned}
\end{equation}
Finally, we have for $IV$
\begin{equation} \label{eq:fourthterm}
\begin{aligned}
IV\leq \left(\sum\limits_{j=1}^{N_h}\lambda^{\alpha-1}_{h,j} (\frac{\tau^3\lambda_{h,j}^{3/2}}{6})^2(v_2,\phi_{h,j})^2\right)^{1/2}=\frac{\tau^3}{6}\| v_2\|_{h,\alpha+2}.
\end{aligned}
\end{equation}
Using (\ref{eq:firstterm2}), (\ref{eq:secondterm3}), (\ref{eq:thirdterm}), and (\ref{eq:fourthterm}), gives for (\ref{eq:firstterm})
\begin{equation*} 
\begin{split}
\vertiii{ (E_h - M)V}_{h,\alpha} &\leq \frac{1}{3}\tau^3\|v_1\|_{h,\alpha+3}+ \frac{7}{12}\tau^3\|  v_2\|_{h,\alpha+2}
\leq \tau^3\vertiii{ V}_{h,\alpha+3}.
\end{split}
\end{equation*}
The above bound completes the proof of the lemma.
\end{proof}

\begin{theorem} 
\label{the:fullydisces}
Let $X^n=[X_1^n, X_2^n]^T$ be the numerical approximation of (\ref{eq:dgsde}) by the stochastic position Verlet method (\ref{eq:stochasticsv}). Assume that $\| U_{h,0} \|_{L_2({\Omega, {\cal H}_h^\beta})}<\infty$ for some $\beta\geq 0$ and that  f satisfies (\ref{eq:assumf}). Then
there exists  a constant $C>0$ depending on  $T, \| U_{h,0} \|_{L_2({\Omega, {\cal H}_h^\beta})}$, and $ \|\Lambda^{(\beta-1)/2}Q^{1/2}\|_{\rm HS}$, but   independent of $h$ and $\tau$,  under the CFL condition \eqref{eq:CFL}, such that:
\begin{enumerate}[i.]
\item   If $W(t)$ satisfies Assumption \ref{assum:noise} for some $\beta\geq 0$, then
\begin{equation} \label{eq:firstcompes}
\| u_{h,1}(t_n)-X_1^{n}\|_{L_2\left(\Omega, \dot{H}_h^0\right)} \leq C\tau^{\min(\frac{2}{3}\beta,1)}.
\end{equation}
\item If $W(t)$ satisfies Assumption \ref{assum:noise} for some $\beta\geq 1$, then
\begin{equation} \label{eq:secondcompes}
\| u_{h,2}(t_n)-X_2^n\|_{L_2\left(\Omega, \dot{H}^0_h\right)} \leq C\tau^{\min(\frac{2}{3}(\beta-1),1)}.
\end{equation}
\end{enumerate}
\end{theorem}

\begin{proof} Let $\alpha\in[ 0,1]$. By the stability of the semi-discrete mild solution (\ref{eq:mildstab}) at discrete times $t_n$ and the stability of the approximate solution (\ref{eq:stochsvstability}), and by the triangle inequality,  we have 
\begin{equation} \label{eq:firstcomp0}
\begin{aligned}
\| U_h(t_{n}) - X^{n}\|_{L_2(\Omega, {\cal H}_h^\alpha)}^2& \leq 2 \| U_h(t_{n})\|_{L_2(\Omega, {\cal H}_h^\alpha)}^2 + 2\| X^{n}\|_{L_2(\Omega, {\cal H}_h^\alpha)}^2\\
&\leq C\left(\left( \| X^0\|_{L_2(\Omega, {\cal H}_h^\alpha)}^2 +1\right)+ \| \Lambda^{(\alpha -1)/2}Q^{1/2}\|_{\rm HS}^2\right).
\end{aligned}
\end{equation}
  Recall that the mild solution (\ref{eq:discretemild}) is given at the discrete times $t_n=n\tau$ by 
\begin{equation*} \label{eq:exactsol}
U_h(t_{n})=E_h(\tau) U_{h}(t_{n-1}) +  \int_{t_{n-1}}^{t_{n}} E_h(t_{n}-s)F(U_h(s)){\rm d}s + \int_{t_{n-1}}^{t_{n}} E_h(t_{n}-s)BP_h{\rm d}W(s).
\end{equation*}
Subtracting  \eqref{eq:stochasticsv_1storder} from the above equation, and adding and subtracting $MU_h(t_{n-1})$ ,  we get 
\begin{equation*} 
\begin{aligned} 
U_h(t_{n})-X^{n}=\,&(E_h(\tau)-M)U_h(t_{n-1}) + M(U_h(t_{n-1})-X^{n-1})\\
& + \int_{t_{n-1}}^{t_{n}} E_h(t_{n}-s)P_hF(U_h(s)){\rm d}s  - \tau  DP_hf(X^{n-1/2}_1)\\
&+\int_{t_{n-1}}^{t_{n}}( E_h(t_{n}-s)B-D)P_h{\rm d}W(s).
\end{aligned}
\end{equation*}
Letting ${\rm Err}_{U}^n= U_h(t_n)- X^n$, Err$_{\rm d}^n=(E_h(\tau)-M)U_h(t_n)$, ${\rm Err}_{\rm non}^n =  \int_{t_{n-1}}^{t_{n}} E_h(t_{n}-s)P_hF(U_h(s)){\rm d}s  - \tau  DP_hf(X^{n-1/2}_1)$, and Err$_s^n=\int_{t_{n-1}}^{t_{n}}( E_h(t_{n}-s)B-D)P_h{\rm d}W(s)$, we have 
\begin{equation*}
\begin{split} 
{\rm Err}_{U}^n &= {\rm Err}_{\rm d}^{n-1} + M {\rm Err}_{U}^{n-1} + {\rm Err}_{\rm non}^{n} + {\rm Err}_{\rm s}^{n}\\
& = \sum\limits_{j=1}^{n}M^{n-j}{\rm Err}_{\rm d}^{j-1} + \sum\limits_{j=1}^{n}M^{n-j}{\rm Err}_{\rm non}^{j}+ \sum\limits_{j=1}^{n}M^{n-j}{\rm Err}_{\rm s}^j,
\end{split}
\end{equation*}
since ${\rm Err}_U^0 = U_{h,0}-X^0= 0$, see \eqref{eq:svnsteps2}.
Using the discrete norm   (\ref{eq:discretemnorm}) and that $\mathbb{E}\left[\langle{\rm Err}_s^j,{\rm Err}_s^k\rangle_{m,\alpha}  \right]=0$, for $j\neq k$, we obtain
\begin{equation*}
\begin{split} 
\mathbb{E}[\|{\rm Err}_U^{n}\|_{m,\alpha}^2]
&\leq  3n \sum\limits_{j=1}^{n}\mathbb{E}\left[\left\|M^{n-j}{\rm Err}_{\rm d}^{j-1}\right\|_{m,\alpha}^2 \right]+3n\sum\limits_{j=1}^{n}\mathbb{E}\left[\left\|M^{n-j}{\rm Err}_{\rm non}^j\right\|_{m,\alpha}^2 \right]\\
&\quad +3 \sum\limits_{j=1}^{n}\mathbb{E}\left[\left\|M^{n-j}{\rm Err}_{\rm s}^j\right\|_{m,\alpha}^2 \right].
\end{split}
\end{equation*}
Employing (\ref{eq:boundm}) we have for the above equation
\begin{equation*}
\begin{split} 
\mathbb{E}[\|{\rm Err}_U^{n}\|_{m,\alpha}^2] 
&\leq 3\sum\limits_{j=1}^{n}\left(  n \mathbb{E}\left[\left\|{\rm Err}_{\rm d}^{j-1}\right\|_{m,\alpha}^2 \right]+n\left[\left\|{\rm Err}_{\rm non}^j\right\|_{m,\alpha}^2 \right]+ \mathbb{E}\left[\left\|{\rm Err}_{\rm s}^j\right\|_{m,\alpha}^2 \right]\right).
\end{split}
\end{equation*}
By the norm equivalence (\ref{eq:normeq}), we have 
\begin{equation} \label{eq:erroreq2}
\begin{split}
\mathbb{E}\left[\vertiii{{\rm Err}_U^{n}}_{h,\alpha}^2\right]& \leq C \sum\limits_{j=1}^{n}\left(n \mathbb{E}\left[\vertiii{{\rm Err}_{\rm d}^{j-1}}_{h,\alpha}^2 \right]+n\mathbb{E}\left[\vertiii{{\rm Err}_{\rm non}^j}_{h,\alpha}^2 \right]
+ \mathbb{E}\left[\vertiii{{\rm Err}_{\rm s}^j}_{h,\alpha}^2 \right]\right)\\
&= C\left({\rm Err}_1 +{\rm Err}_2 +{\rm Err}_{3}\right).
\end{split}
\end{equation}
Using  the estimates (\ref{eq:EminusM}) and  (\ref{eq:mildstab}) gives for ${\rm Err}_1$ 
\begin{equation} \label{eq:Ies}
\begin{split}
{\rm Err}_1&\leq \tau^6n  \mathbb{E}\left[\vertiii{ U_{h}(t_{j-1})}_{h,\alpha+3}^2\right]\leq T\tau^5\sup_{t_{j-1}\in[0,T]}\mathbb{E}\left[\vertiii{ U_{h}(t_{j-1})}_{h,\alpha+3}^2\right]\leq C\tau^5.
\end{split}
\end{equation}

By the definition of the $\vertiii{\cdot}_{h,\alpha}$ norm, we have for ${\rm Err}_2$
\begin{equation*} 
\begin{split}
{\rm Err}_2 & = n\left( \mathbb{E}\left[\left\| \int_{t_{j-1}}^{t_j}\Lambda_h^{-1/2}S_h(t_j-s)P_hf(u_{h,1}(s)){\rm d}s - \frac{\tau^2}{2}P_h f(X_1^{j-1/2})\right\|_{h,\alpha}^2\right]\right.\\
& \left. \quad\quad \quad \quad  +\mathbb{E}\left[ \left\|\int_{t_{j-1}}^{t_j}C_h(t_j-s)P_hf(u_{h,1}(s)){\rm d}s - \tau P_h f(X_1^{j-1/2}) \right\|_{h,\alpha-1}^2\right] \right)\\
& =n\left( {\rm Err}_{[2,1]} + {\rm Err}_{[2,2]}\right).
\end{split}
\end{equation*}
Using the triangle inequality,  Taylor's theorem for $\Lambda_h^{-1/2}S_h(t_{j}-s)$ up to first order and (\ref{eq:assumf}) for f, we obtain for ${\rm Err}_{[2,1]}$
\begin{equation*}
 \begin{split} 
 {\rm Err}_{[2,1]}  \leq 2C &\left(\tau \int_{t_{j-1}}^{t_j}|t_j-s|^2 \mathbb{E}\left[ 1+\| u_{h,1}(s)\|^2_{h,\alpha}\right] {\rm d}s 
+\frac{\tau^4}{4}\mathbb{E}\left[1+\vertiii{ X_1^{j-1/2}}_{h,\alpha}^2\right] \right)\\
  \leq 2C &\left(\tau \int_{t_{j-1}}^{t_j}|t_j-s|^2 \sup\limits_{t\in[0,T]}\mathbb{E}\left[ 1+\| u_{h,1}(t)\|^2_{h,\alpha}\right] {\rm d}s \right.\\
&\left. \quad +\frac{\tau^4}{4}\mathbb{E}\left[1+\vertiii{X_1^{j-1/2}}_{h,\alpha}^2\right] \right)\\
&\leq C \tau^4 \left( \mathbb{E}\left[1+\vertiii{ X^0}_{h,\alpha}^2\right] + \|\Lambda^{(\alpha-1)/2}Q^{1/2}\|^2_{\rm HS} \right),
\end{split}
\end{equation*}
by estimates (\ref{eq:mildstab}) and (\ref{eq:stochsvstability}).

Adding and subtracting $ P_hf(u_{h,1}(s)){\rm d}s$ and using the triangle inequality gives 
\begin{equation*} 
\begin{split}
 {\rm Err}_{[2,2]}\leq & \,2  \left(\mathbb{E}\left[ \left\|\int_{t_{j-1}}^{t_j}(C_h(t_j-s)-I)P_hf(u_{h,1}(s)){\rm d}s\right\|_{h,\alpha-1}^2\right]\right.\\
& \left. \quad \quad\quad   +\mathbb{E}\left[ \left\|  \int_{t_{j-1}}^{t_j}P_hf(u_{h,1}(s)){\rm d}s- \tau P_h f(X_1^{j-1/2}) \right\|_{h,\alpha-1}^2\right] \right)\\
  = &\,2  \left(  {\rm Err}_{[2,2]}^1 +  {\rm Err}_{[2,2]}^2\right).
\end{split}
\end{equation*}
Using  Taylor's theorem for $|\cos{((t_{j}-s)\sqrt{\lambda_{h,j}})}-1|\leq (t_{j}-s)\sqrt{\lambda_{h,j}}$, (\ref{eq:assumf}) for f and the stability estimate (\ref{eq:mildstab}), we get for $ {\rm Err}_{[2,2]}^1$
\begin{equation*} 
\begin{split}
 {\rm Err}_{[2,2]}^1\leq &\,\tau \int_{t_{j-1}}^{t_j}|t_j-s|^2{\rm d}s\sup\limits_{t\in[0,T]}\mathbb{E}\left[1+\left\|u_{h,1}(t))\right\|_{h,\alpha}^2\right] \leq C  \tau^4.
\end{split}
\end{equation*}
Adding and subtracting $ \int_{t_{j-1}}^{t_j}P_hf(u_{h,1}(t_{j-1})+\frac{\tau}{2}u_{h,2}(t_{j-1})){\rm d}s$, we obtain for $ {\rm Err}_{[2,2]}^2$
\begin{equation*} 
\begin{split}
 {\rm Err}_{[2,2]}^2 \leq 
&\,2\mathbb{E}\left[\tau \int_{t_{j-1}}^{t_j} \left\| P_h\left(f(u_{h,1}(s)) -f\left(u_{h,1}(t_{j-1})+\frac{\tau}{2}u_{h,2}(t_{j-1})\right)\right) \right\|_{h,\alpha-1}^2{\rm d}s\right]  \\ 
&+2\mathbb{E}\left[ \left\|\tau P_hf\left(u_{h,1}(t_{j-1})+\frac{\tau}{2}u_{h,2}(t_{j-1})\right) -  \tau P_h f(X_1^{j-1/2}) \right\|_{h,\alpha-1}^2\right].
\end{split}
\end{equation*}
Applying (\ref{eq:assumf}) for both terms above and  the fact  that $\|\Lambda_h^{(\alpha -1)/2}u\|\leq \|u\|$, $\alpha\in[0,1]$ for the first term, gives 
\begin{equation*} 
\begin{split}
 {\rm Err}_{[2,2]}^2 
\leq &\,4 C\mathbb{E}\left[ \tau \int_{t_{j-1}}^{t_j} \left(\left\| u_{h,1}(s) -u_{h,1}(t_{j-1})\right\|^2+\left\|\frac{\tau}{2}u_{h,2}(t_{j-1}) \right\|^2 \right){\rm d}s\right]  \\ 
&+ 4 C\tau^2\mathbb{E}\left[ \left\|u_{h,1}(t_{j-1})-X_1^{j-1}\right\|^2_{h,\alpha-1} +\left\|\frac{\tau}{2}u_{h,2}(t_{j-1}) -\frac{\tau}{2}X_2^{j-1}\right\|_{h,\alpha-1}^2\right].
\end{split}
\end{equation*}
By H\"{o}lder's continuity (\ref{eq:holdersol}) for the first term in the above inequality, we have
\begin{equation*} 
\begin{split}
 {\rm Err}_{[2,2]}^2 \leq
& \, 4C  \tau\int_{t_{j-1}}^{t_j}|s-t_{j-1}|^2{\rm d}s
+ 4C \frac{\tau^3}{4} \int_{t_{j-1}}^{t_j}\sup\limits_{t\in[0,T]}\mathbb{E}\left[ \left\| u_{h,2}(t) \right\|_{h,\alpha-1}^2\right]{\rm d}s  \\ 
&+ 4C\tau^2 \mathbb{E}\left[ \left\|u_{h,1}(t_{j-1}) -X_1^{j-1} \right\|_{h,\alpha-1}^2 + \frac{\tau^2}{4} \left\|u_{h,2}(t_{j-1}) -X_2^{j-1} \right\|_{h,\alpha-1}^2\right].
\end{split}
\end{equation*}
Using the stability estimate (\ref{eq:mildstab}) for the second term,  the inverse estimate (\ref{eq:inverse}) together with the CFL condition (\ref{eq:CFL}) for the third term, we get for $ {\rm Err}_{[2,2]}^2$
\begin{equation*} 
\begin{split}
 {\rm Err}_{[2,2]}^2 \leq &\,  C\tau^4
+ 4 \tau^2\mathbb{E}\left[ \left\|u_{h,1}(t_{j-1}) -X_1^{j-1} \right\|_{h,\alpha-1}^2 +\left\|u_{h,2}(t_{j-1}) -X_2^{j-1} \right\|_{h,\alpha-2}^2\right]\\
=& \,  C\tau^4 +  4C\tau^2\mathbb{E}\left[\vertiii{{\rm Err}_U^{j-1}}_{h,\alpha}^2 \right].
\end{split}
\end{equation*}
Combining the estimates for ${\rm Err}_{[2,2]}^1$ and ${\rm Err}_{[2,2]}^2$, we have for ${\rm Err}_{[2,2]}$
\begin{equation*}
{\rm Err}_{[2,2]}\leq   C\tau^4 + 4C\tau^2\mathbb{E}\left[\vertiii{{\rm Err}_U^{j-1}}_{h,\alpha}^2 \right].
\end{equation*}
Then combining this and  the estimate for  ${\rm Err}_{[2,1]}$, we obtain for ${\rm Err}_{2}$
\begin{equation} \label{eq:err2}
\begin{split} 
{\rm Err}_{2}
&\leq C \tau^3 \left( \mathbb{E}\left[1+\vertiii{ X^0}_{h,\alpha}^2\right] + \|\Lambda^{(\alpha-1)/2}Q^{1/2}\|^2_{\rm HS} \right.\\
& \left.+ 4\tau \sum\limits_{j=1}^{n}\mathbb{E}\left[\vertiii{{\rm Err}_U^{j-1}}_{h,\alpha}^2 \right]\right).
\end{split}
\end{equation}

By It$\hat{\rm o}$'s isometry (\ref{eq:itosisometry}) we have for ${\rm Err}_3$ 
\begin{equation*} 
\begin{split}
{\rm Err}_3 & =\int_{t_{j-1}}^{t_{j}}\| \Lambda_h^{\alpha/2}(\Lambda_h^{-1/2}S_h(t_{j}-s)-\frac{\tau}{2}I)P_hQ^{1/2} \|^2_{{\rm HS}}{\rm d}s\\
& \quad+\int_{t_{j-1}}^{t_{j}}\|\Lambda_h^{(\alpha-1)/2} (C_h(t_{j}-s)-I)P_hQ^{1/2} \|^2_{{\rm HS}}{\rm d}s.
\end{split}
\end{equation*}
Using  triangle inequality for the first term, and Taylor's theorem for $\Lambda_h^{-1/2}S_h(t_{j}-s)$ and $C_h(t_j-s)$ up to first order, we get 
\begin{equation*} 
\begin{split}
{\rm Err}_{3}& \leq  3\int_{t_{j-1}}^{t_{j}}|t_{j} -s|^2 \|\Lambda_h^{\alpha/2} P_hQ^{1/2}\|_{\rm HS}^2 {\rm d}s + 2\frac{\tau^3}{4}\|\Lambda_h^{\alpha/2}P_hQ^{1/2}\|_{\rm HS}\\
&\leq \tau^3\|\Lambda_h^{\alpha/2}P_hQ^{1/2}\|_{\rm HS}^2.
\end{split}
\end{equation*}

By estimates (\ref{eq:Ies}), (\ref{eq:err2}) and  the above estimate for ${\rm Err}_{3}$, we have for (\ref{eq:erroreq2})
\begin{equation*}
\begin{split} 
\mathbb{E}\left[\vertiii{{\rm Err}_U^{n}}_{h,\alpha}^2\right] &\leq  C\left(\tau^4  +  \tau^2 \mathbb{E}\left[ 1+\vertiii{X^0}^2_{h,\alpha}\right]\right. \\
&\left.+ \tau^2\|\Lambda^{\alpha/2}Q^{1/2}\|_{\rm HS}^2
+\tau \sum\limits_{j=1}^{n}\mathbb{E}\left[\vertiii{{\rm Err}_U^{j-1}}_{h,\alpha}^2 \right]\right).
\end{split}
\end{equation*} 
By applying the discrete  Gronwall's inequality to the above, we get
\begin{equation*} 
\begin{split}
\mathbb{E}\left[\vertiii{{\rm Err}_U^{n}}_{h,\alpha}^2\right]\leq C &\left( \tau^4 +\tau^2 \mathbb{E}\left[ 1+\vertiii{X^0}^2_{h,\alpha}\right] 
+  \tau^2 \| \Lambda^{\alpha/2}Q^{1/2}\|_{\rm HS}^2\right).
\end{split}
\end{equation*}
By interpolation between (\ref{eq:firstcomp0}) and the above estimate and taking square roots, we have for $\beta\geq \alpha$
\begin{equation*} 
\begin{split}
\| {\rm Err}_U^{n}\|_{L_2(\Omega, {\cal H}_h^\alpha)} \leq & C\left(\tau^{\frac{2}{3}(\beta-\alpha)} (\| X^0\|_{L_2(\Omega, {\cal H}^\beta_h)}  +1)\right.\\
&\left. \qquad \quad +  \tau^{\min(\beta-\alpha,1)} \| \Lambda_h^{(\beta-1)/2}Q^{1/2}\|_{\rm HS}\right).
\end{split}
\end{equation*} 
The proof of estimate (\ref{eq:firstcompes}) follows from setting $\alpha=0$ and the proof of estimate  (\ref{eq:secondcompes}) follows from setting $\alpha=1$ in the above bound.
\end{proof}

We now state the strong convergence rates for the fully discrete  stochastic wave equation (\ref{wave}).
\begin{theorem} 
\label{the:spaceandtime}
Let $U=[u_1,u_2]^T$ and $X=[X^n_1, X^n_2]^T$ be given by (\ref{mildsolution}) and (\ref{eq:svnsteps2}), respectively. Also let  the assumptions of Theorems \ref{the:dgmethodes} and \ref{the:fullydisces} be fulfilled. Then the following estimates hold, under the CFL condition \eqref{eq:CFL}, at discrete times $t_n\in[0,T]$
\begin{equation*} 
\| u_1(t_n)-X^n_1\|_{L_2(\Omega,\dot{ H}^0)}\leq C(\tau^{\min(\frac{2}{3}\beta,1)}+h^{\min(\frac{p+1}{p+2}\beta,p+1)}),\qquad \beta \geq 0,
\end{equation*}

\begin{equation*} 
\| u_2(t_n)-X^n_2\|_{L_2(\Omega,\dot{ H}^0)}\leq C(\tau^{\min(\frac{2}{3}(\beta-1),1)}+h^{\min(\frac{p+1}{p+2}(\beta-1),p+1)}), \qquad \beta \geq 1.
\end{equation*}

\end{theorem}
\begin{proof} 
The proof follows from Theorems \ref{the:dgmethodes} and \ref{the:fullydisces} and triangle inequality.
\end{proof}

\section{Energy conservation}
\label{sec:energy}
In this section, we state bounds for the energy (or Hamiltonian) of the fully discrete  stochastic wave equation \eqref{eq:dgsde} and the linear analogue of it. We consider a trace-class $Q$-Wiener process, i.e., ${\rm Tr}(Q)=\|Q^{1/2}\|_{\rm HS}^2<\infty$ and the nonlinearity $f(u)=-V'(u)$ for a smooth potential V.  The ``Hamiltonian'' function for the dG approximation $U_h$  is defined on ${\cal H}_h^1 = \dot{H}_h^1\times \dot{H}_h^0$ as 
\begin{equation*} 
\begin{split}
H(U_h) 
= \frac{1}{2}\|u_{h,1}\|_*^2 + \frac{1}{2}\|u_{h,2}\|^2 +  \int_{\cal D}V(u_{h,1}){\rm d}x,
\end{split}
\end{equation*}
where the broken norm $\|\cdot\|_*$ is defined in \eqref{eq:starnorm}.
In the following proposition we state the trace formula for the dG semi-discrete exact solution \eqref{eq:discretemild}.
\begin{prop} 
\label{prop:traceformula}
Let $f(u)=-V'(u)$ for a smooth potential $V\colon \mathbb{R} \rightarrow \mathbb{R}$, ${\rm Tr}(Q)<\infty $, and  the Hamiltonian H be defined as above. Then, the dG approximation to the stochastic wave equation \eqref{wave}, $U_h(t)$ in \eqref{eq:discretemild}, satisfies the  trace formula 
\begin{equation} 
\mathbb{E}\left[ H(U_h(t))\right]= \mathbb{E}\left[ H(U_{h,0})\right] + \frac{1}{2}t{\rm Tr}(P_hQP_h), \quad t\in[0,T].
\end{equation}
\end{prop}
\begin{proof}
The proof  follows from Proposition 5 in \cite{anton2015} by taking into account that $\|u\|_* = \| \Lambda_h^{1/2} u\|$.
\end{proof}

We now derive a bound for the Hamiltonian of the temporal approximation to the nonlinear stochastic wave equation. 
\begin{theorem} 
\label{thm:energynon}
Let $f$ and $W$ be as in Proposition \ref{prop:traceformula}. Also let the assumptions in Lemma \ref{lem:stabilitystoch} be fulfilled for $\beta=1$. The numerical approximation of \eqref{eq:dgsde} by the stochastic position Verlet method \eqref{eq:svnsteps2} satisfies, under the CFL condition \eqref{eq:CFL} the following bound for the Hamiltonian H 
\begin{equation} \label{eq:discreteH}
\mathbb{E}\left[ H(X^n)\right]\leq  \mathbb{E}\left[ H(X^0)\right]+  C\exp(2Ct_n)( t_n \| Q^{1/2}\|_{\rm HS}^2 +t_n^2),
\end{equation} 
for $0\leq t_n\leq T$ and $\hat{C}$ independent of $h, \tau$, and $T$.
\end{theorem}
\begin{proof} 
We first note that 
\begin{equation} \label{eq:HminusH0}
\begin{split}
\mathbb{E}\left[ H(X^n)\right] - \mathbb{E}\left[ H(X^0)\right] 
 =\,& \frac{1}{2} \mathbb{E}\left[ \vertiii{X^n}_{h,1}^2 \right] -\frac{1}{2}\mathbb{E}\left[ \vertiii{X^0}_{h,1}^2\right]\\
& + \mathbb{E}\left[\int_{\cal D}(V(X_1^n) - V(X_1^0)){\rm d}x\right].
\end{split}
\end{equation}
The first term in the above inequality is bounded from estimate \eqref{eq:stochsvstability} for $\beta=1$ 
\begin{equation*} 
 \mathbb{E}\left[\vertiii{X^n}_{h,1}^2\right] \leq C\exp(2Ct_n)\left( t_n \|Q^{1/2}\|_{\rm HS}^2 + t_n^2\right).
\end{equation*}
For the third term in \eqref{eq:HminusH0}, using the mean value theorem we obtain 
\begin{equation*} 
\begin{split} 
 \mathbb{E}\left[\|V(X_1^n) - V(X_1^0)\|_{L_1({\cal D})}\right] \leq  \,& \mathbb{E}\left[\|V(X_1^n) - V(X_1^0)\|_{L_2({\cal D})}\right] \\ 
\leq \, & C \left(\mathbb{E}\left[\|V'(\xi)(X_1^n - X_1^0)\|^2_{L_2({\cal D})}\right]\right)^{1/2}.
\end{split}
\end{equation*}
Recalling that $V'(u) = -f(u)$, we have by \eqref{eq:assumf} $\|V'(\xi) \|=\|f(\xi)\| \leq C(1+\| \xi\|)$. Since $\xi\in (X^0,X^n)$, we have by triangle inequality and estimate  \eqref{eq:stochsvstability} for $\beta=0$  
\begin{equation*} 
\begin{split} 
 \mathbb{E}\left[\|V(X_1^n) - V(X_1^0)\|_{L_1({\cal D})}\right] \leq  C\exp(2Ct_n)\left( t_n \|\Lambda^{-1/2}Q^{1/2}\|^2_{\rm HS} + t_n^2\right).
\end{split}
\end{equation*}
The above completes the proof of \eqref{eq:discreteH}.
\end{proof}

In the case of the linear stochastic wave equation \eqref{eq:linearsto}, the discrete energy of the temporal approximation is given by,
\begin{equation*}
{\cal E}^n= \frac{1}{2}\|\Lambda_h^{1/2} Y^{n}_1\|^2 + \frac{1}{2}\|Y^n_2\|^2
\end{equation*}
 where we recall that  $Y^n=[Y_1^n,Y^n_2]^T$ is the numerical approximation in \eqref{eq:linearprob}.
We also introduce a so-called modified energy
\begin{equation*}
{\cal E}^n_m = \frac{1}{2}\left\|( I-\frac{\tau^2}{4}\Lambda_h)^{-1/2}\Lambda_h^{1/2} Y^{n}_1\right\|^2 + \frac{1}{2}\left\|Y^n_2\right\|^2= \| Y^n\|^2_{m,1}.
\end{equation*}

 Theorem 5.1 in \cite{cohen2013} proves that the  expected value of the energy of the exact semi-discrete solution to the linear stochastic equation grows linearly with time t.
 We observe that the SVM applied to this problem preserves the linear growth of the expected value of the modified energy ${\cal E}^n_m$ with the time t. 
\begin{theorem}
\label{thm:linearen}
 Let the assumptions in Lemma \ref{lem:stabilitystoch} be fulfilled for $\beta=1$. Then under the CFL condition \eqref{eq:CFL},  the expected value of the modified energy ${\cal E}^n_m$ satisfies 
\begin{equation*} 
\mathbb{E}\left[ {\cal E}^n_m\right] = \expect{ \left\| Y^0\right\|_{m,1}^2}+ t_n\sum_{k = 0}^\infty\left\|D P_hQ^{1/2}e_k\right\|_{m,1}^2.
\end{equation*}
Further, the expected value of the discrete energy ${\cal E}^n$
is bounded, at $t_n=n\tau$,  by
\begin{equation*} 
\mathbb{E}\left[ {\cal E}^n\right]\leq C\left( \mathbb{E}\left[ \vertiii{Y^0}_{h,1}^2\right] +  t_n\|Q^{1/2}\|^2_{\rm HS}\right), 
\end{equation*}
where $C$ is a constant independent of $h$, $\tau$ and $n$.
\end{theorem}
\begin{proof} 
 We have for \eqref{eq:linearprob} by using the It\^{o}'s isometry \eqref{eq:itosisometry} and \eqref{eq:boundm}
\begin{equation*}
\begin{split}    
\expect{\|Y^n\|_{m,1}^2} &= \expect{ \|M^n Y^0\|_{m,1}^2}+ \expect{\left\|\sum_{j = 1}^n\int_{t_{j-1}}^{t_{j}} M^{n-j} D P_h dW(s)\right\|_{m,1}^2} \\
&= \expect{ \|Y^0\|_{m,1}^2}+t_n\sum_{k = 0}^\infty\left\|D P_hQ^{1/2}e_k\right\|_{m,1}^2.\\
\end{split}
\end{equation*}
This completes the proof for the bound on ${\cal E}^n_m$.

The proof for the bound on ${\cal E}^n$ follows from estimate \eqref{eq:stablinear} for $\beta=1$.
\end{proof}

 \section{Numerical experiments}
\label{sec:numexp}
We consider the following 1-dimensional Sine-Gordon equation
\begin{equation} \label{eq:example1}
\begin{aligned}
 &{\rm d}\dot{u}=\Delta u {\rm d}t - \sin(u){\rm d}t +{\rm d}W  &&{\rm in}\;\; {\cal D}\times[0,1],\\
&\; u= 0  &&{\rm on } \; \;\partial {\cal D} \times[0,1],\\
&\;  u(\cdot,0) =0,\;\;\dot{u}(\cdot,0)= \sin(\pi x) & &{\rm in} \;\; {\cal D},
\end{aligned}
\end{equation}
where ${\cal D} = (0,1)$. We approximate  the solution of (\ref{eq:example1}) with the dG finite element method (\ref{eq:dgsde}) in space.
Letting $X^{n}$ be the numerical approximation of  (\ref{eq:dgsde}) at discrete times $t_n=\tau n$, $n=1,\dots, N$,  we consider the following integrators:
\begin{enumerate}[i.] 
\item  The stochastic trigonometric method (STM), see \cite{cohen2013},
\begin{equation*} 
X^{n} = E_h(\tau)X^{n-1} + E_h(\tau)BP_hf(X_1^{n-1})\tau + E_h(\tau)BP_h\Delta W^n,
\end{equation*}
where $E_h(\tau)$ is the $C_0$-semigroup defined in (\ref{semigroupdiscrete}) and $B=[0 \,\,  I]^T$. The strong convergence rates for the full discretization are ${\cal O}(\tau^{\min(\beta,1)} + h^{\frac{2}{3}\beta})$, see \cite[Theorem 4]{anton2015}.
\item The semi-implicit Euler-Maruyama method (SEM)
\begin{equation*}
X^{n} =X^{n-1} + \tau A_hX^{n}  + BP_hf(X_1^{n-1})\tau+ BP_h \Delta W^{n}.
\end{equation*}
We refer to \cite{hausblaus2003}  for the mean-square errors of this scheme applied to stochastic parabolic partial differential equations.
\item The stochastic position Verlet method  (SVM) considered here.
\end{enumerate} 

\medskip
For our numerical experiments, we set $Q=\Lambda^{-s}$, $s\in\mathbb{R}$, then using the asymptotic behaviour of the eigenvalues of $\Lambda$, $\lambda_j\sim j^{2/d}$, where $d$ is the dimension of the domain ${\cal D}$, see \cite{Weyl1911}, we get
\begin{equation*} 
\|\Lambda^{(\beta-1)/2}Q^{1/2}\|_{\rm HS}^2=\sum\limits_{j=1}^{\infty}\lambda_j^{(\beta-1-s)}\approx\sum\limits_{j=1}^{\infty}j^{\frac{2}{d}(\beta-1-s)}.
\end{equation*}
The above series converges if and only if $\beta<1 +s-d/2$.

\begin{figure}
\centering
\begin{subfigure}[b]{0.49\textwidth}
\includegraphics[width=\textwidth]{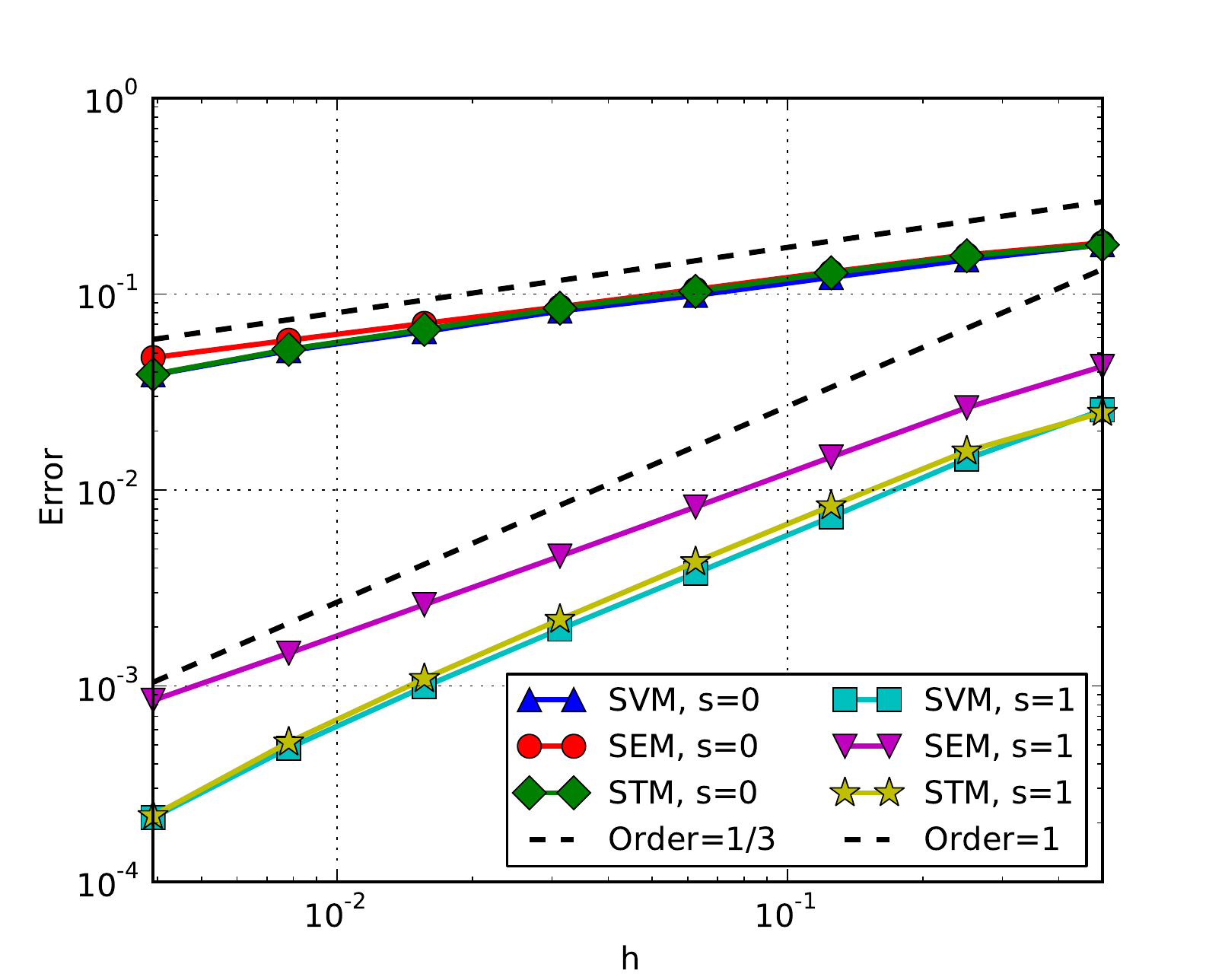}
\caption{}
\end{subfigure}
\begin{subfigure}[b]{0.49\textwidth}
\includegraphics[width=\textwidth]{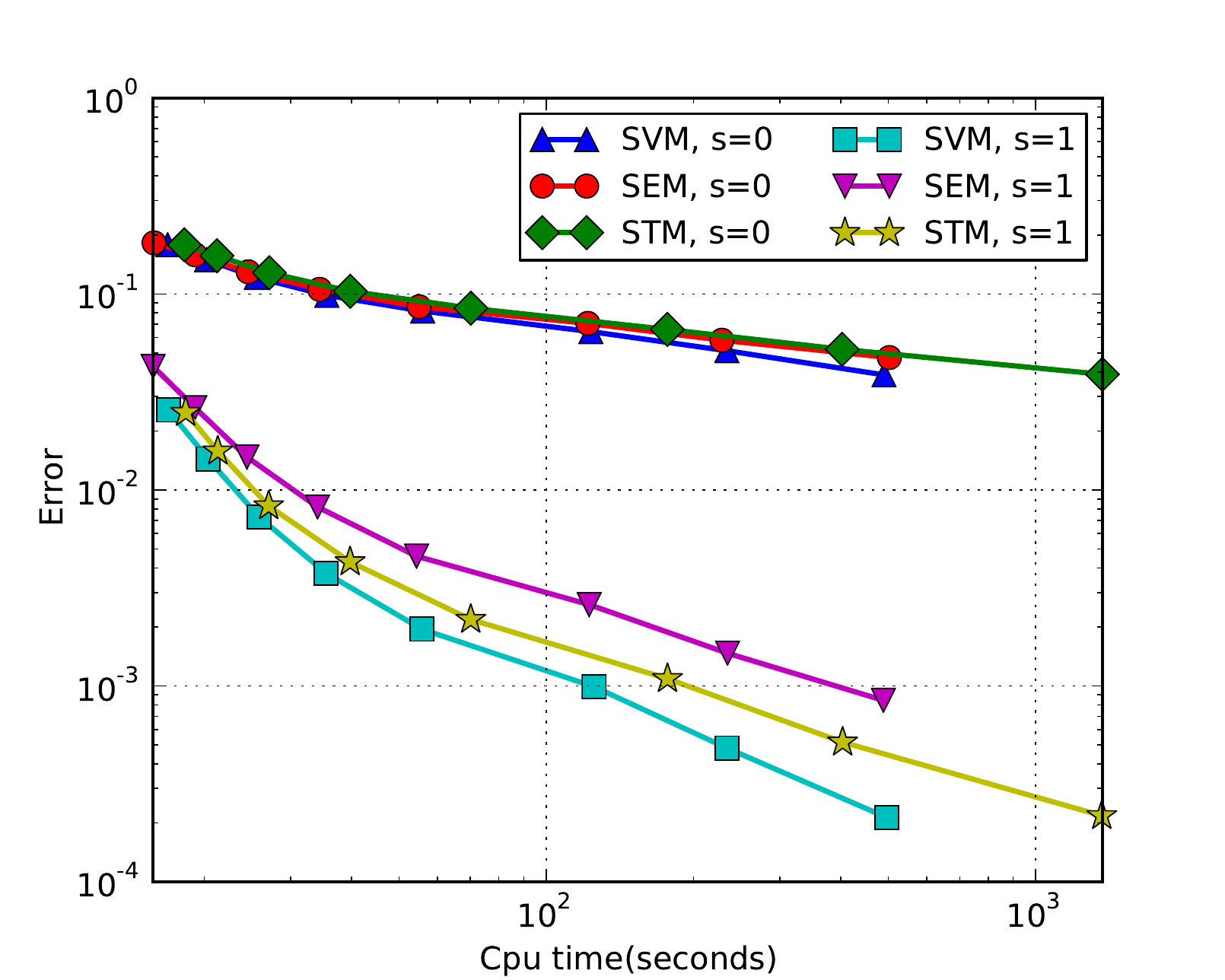}
\caption{}
\end{subfigure}
\caption{Plots (a)  exhibits the spatial rates of convergence of the stochastic position Verlet method (SVM), the stochastic trigonometric method (STM), and the  semi-implicit Euler-Maruyama method (SEM)
 and (b) the  efficiency of these time integrators  for space-time white noise, i.e., $Q=\Lambda^{-s}, s=0,$ and for correlated  noise, i.e., $Q= \Lambda^{-s}, s=1$. }
\label{fig:error}
\end{figure}

 We examine  simultaneously the spatial error and the temporal error for the the displacement for $s=0$, which corresponds to space-time white noise, and $s=1$, which corresponds to correlated noise.  We choose different timesteps $\tau_i=2^{-i}, i=2, \dots,8,$ and different mesh sizes $h_i=2^{-i},\,i=4, \dots ,10$. Furthermore, we take the exact solution to be the dG approximation on  a fine mesh with mesh size $h_{\rm exact}=2^{-10}$ in space and the time integration is done by STM with time step $\tau_{\rm exact}=2^{-12}$. We consider $M=100$ realisations to compute the expected values. Figures \ref{fig:error} display the  strong convergence
 rates and the efficiency of the above numerical schemes. The spatial  mean-square error is defined at final time $T=1$ as 
\begin{equation*}
{\rm Error} =\left( \frac{1}{M}\sum\limits_{m=1}^{M}\| u_{h,m}(\cdot,1) -u_m^{\rm ref}(\cdot,1)\|^2 \right)^{1/2} \approx \left(\mathbb{E}\left[ \| u_h(\cdot,1) -u^{\rm ref}(\cdot,1)\|^2\right]\right)^{1/2}.
\end{equation*}
 We only present the strong numerical error vs. the mesh size $h$ since the spatial convergence rates dominate the convergence rates  of the temporal discretization see Theorem \ref{the:spaceandtime}. We observe that the expected convergence rates are confirmed for SVM and STM for both space-time white noise and correlated noise. Further, SVM is the most efficient of the three methods.

In addition, we investigate the spatial error for  polynomials of total degree equal to $p=2$. Recalling estimate (\ref{stoerror1}), we have that the convergence rate is ${\cal O}(h^{\frac{3}{4}\beta})$ for the displacement. 
 We consider correlated noise by setting $Q=\Lambda^{-1}$. We choose different mesh sizes $h_i=2^{-i},\,i=2, \dots, 7$. We take the exact solution to be the dG approximation on  a fine mesh with mesh size $h_{\rm exact}=2^{-9}$ in space and the time integration is done by the stochastic Verlet scheme with time step $\tau_{\rm exact}=2^{-11}$. Figure \ref{fig:errorp} exhibits the spatial rate of convergence and the efficiency for polynomials of total degree equal to $p=1,2$. Again, we consider $M=100$ realisations to calculate the expected values. We observe that the second-order polynomials are more accurate and efficient when used for the linear problem. 

Finally, we are concerned with the energy results given in Section \ref{sec:energy}. In order to illustrate the results from Section  \ref{sec:energy}, we set $Q=\Lambda^{-s}, s=1,$ and we choose $h=2^{-5}$. For the stochastic Verlet method  we choose the timestep $\tau= 1/6400$. For STM  and SEM we choose timestep $\tau=1/500$. Figure \ref{fig:enresults} displays the  expected value of the Hamiltonian along the numerical solutions of 
\eqref{eq:example1} and the linear analogue of it over the time interval $[0,50]$. Further, we take $M=8000$ samples to approximate the expected values of the energy of the schemes. We observe that  SVM reproduces the linear growth of the exact energy, although this is not expected from Theorems \ref{thm:energynon} and \ref{thm:linearen}. In case of the linear stochastic wave equation, STM preserves the linear growth of the expected value of the energy, see \cite{cohen2013}, as Figure \ref{fig:enresults} (a) verifies.  The unsatisfactory behaviour of SEM has also been previously observed when applied   the nonlinear stochastic wave equation \cite{anton2015} and the linear analogue of it \cite{cohen2013}.

 All the numerical experiments were performed in Python using the finite element  software library Fenics \cite{fenics}. 
\begin{figure}
\centering
\begin{subfigure}[b]{0.49\textwidth}
\includegraphics[width=\textwidth]{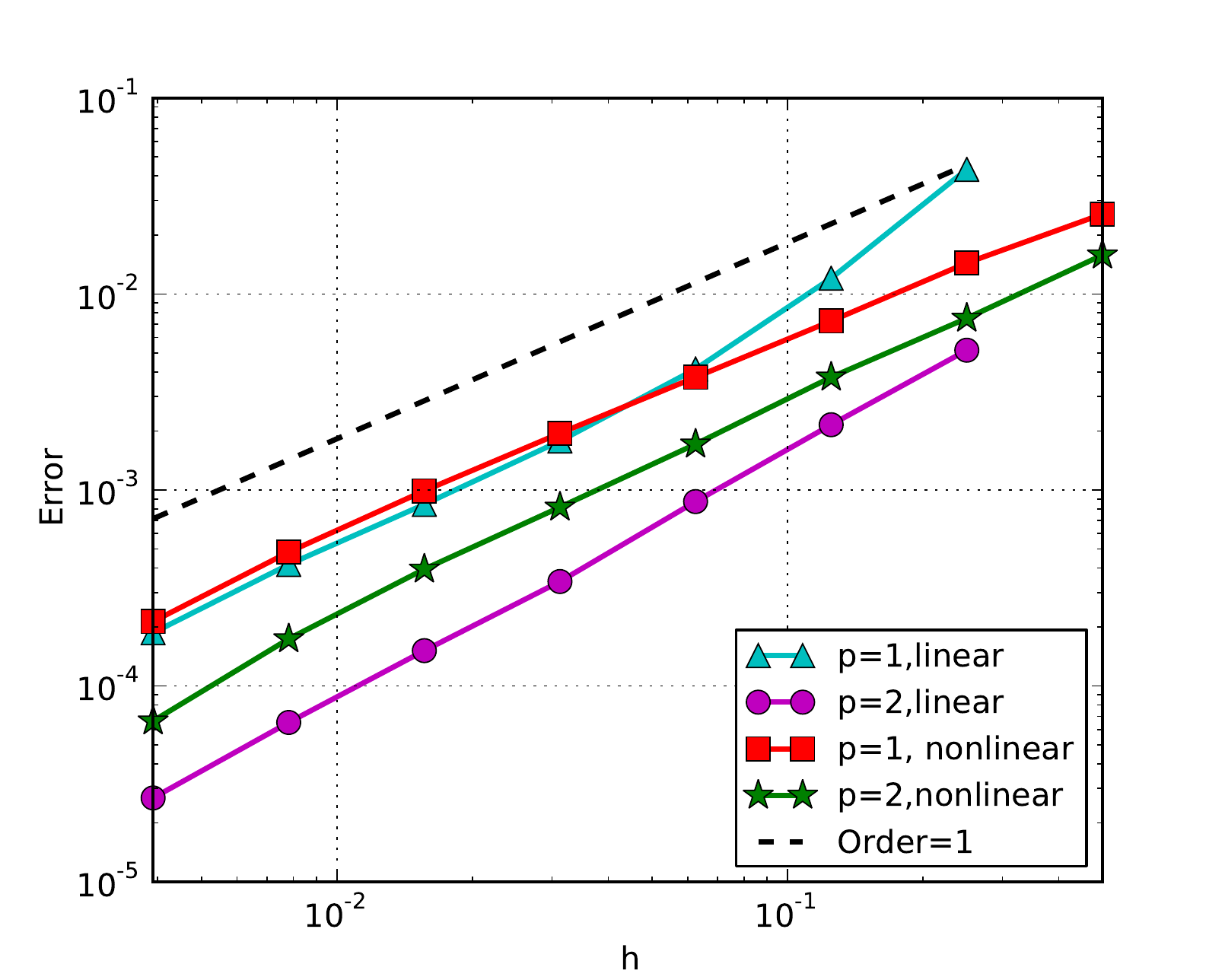}
\caption{}
\end{subfigure}
\begin{subfigure}[b]{0.49\textwidth}
\includegraphics[width=\textwidth]{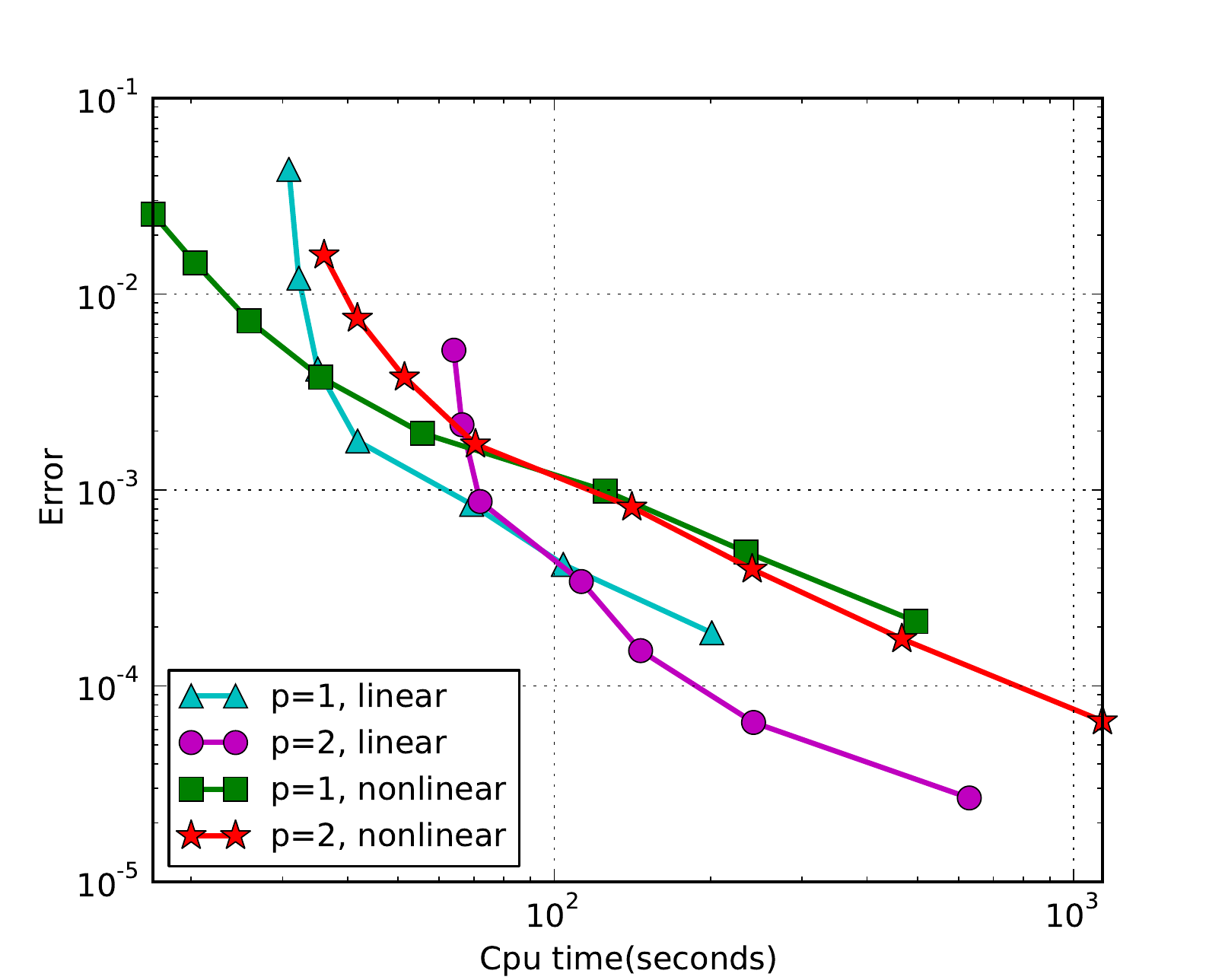}
\caption{}
\end{subfigure}
\caption{Plots (a)  show the strong convergence rates in space  and the efficiency for first-order degree polynomials (p=1) and (b) second-order degree polynomials (p=2) for \eqref{eq:example1} and the linear analogue of it.}
\label{fig:errorp}
\end{figure}

\begin{figure}
\centering
\begin{subfigure}[b]{0.49\textwidth}
\includegraphics[width=\textwidth]{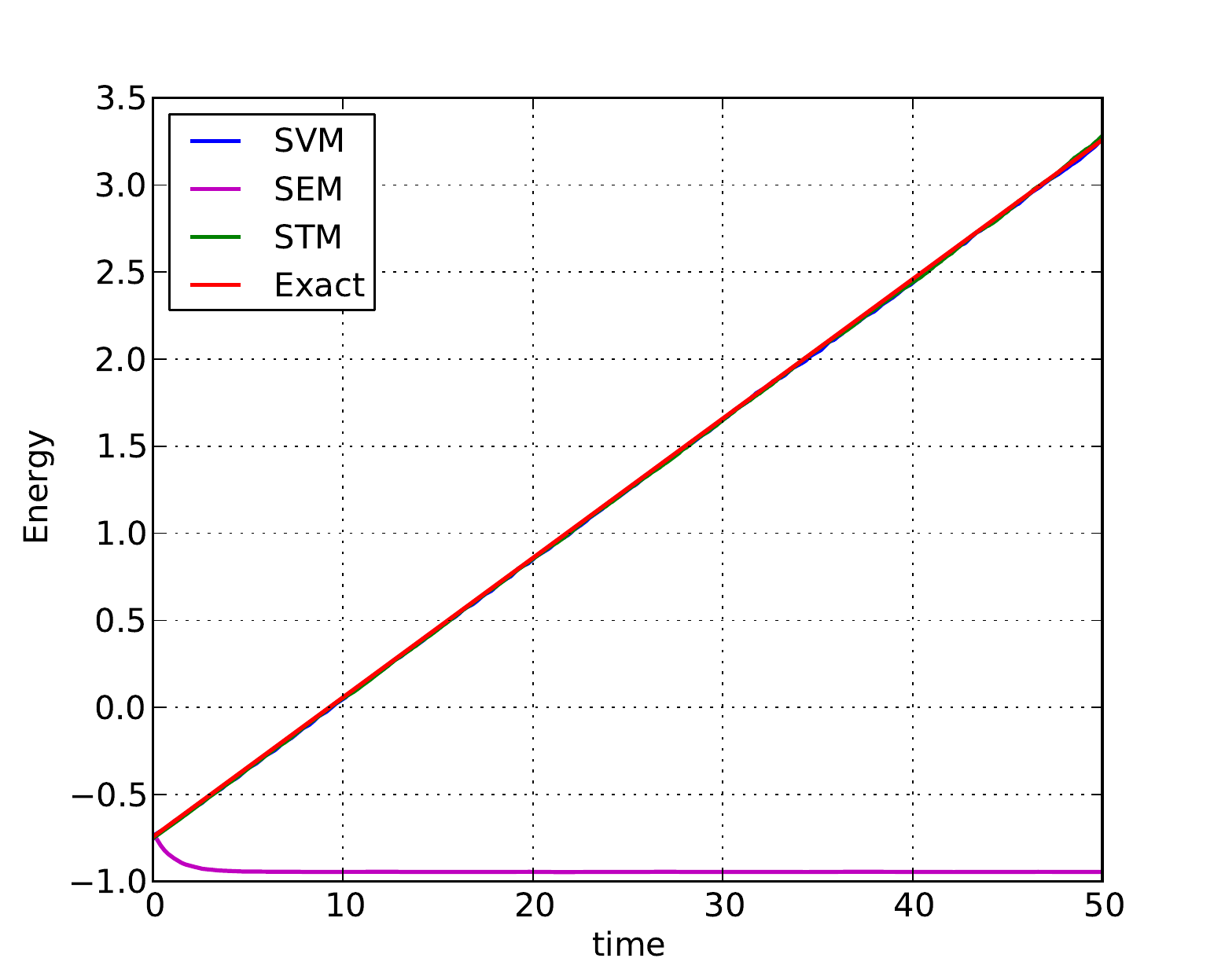}
\caption{}
\end{subfigure}
\begin{subfigure}[b]{0.49\textwidth}
\includegraphics[width=\textwidth]{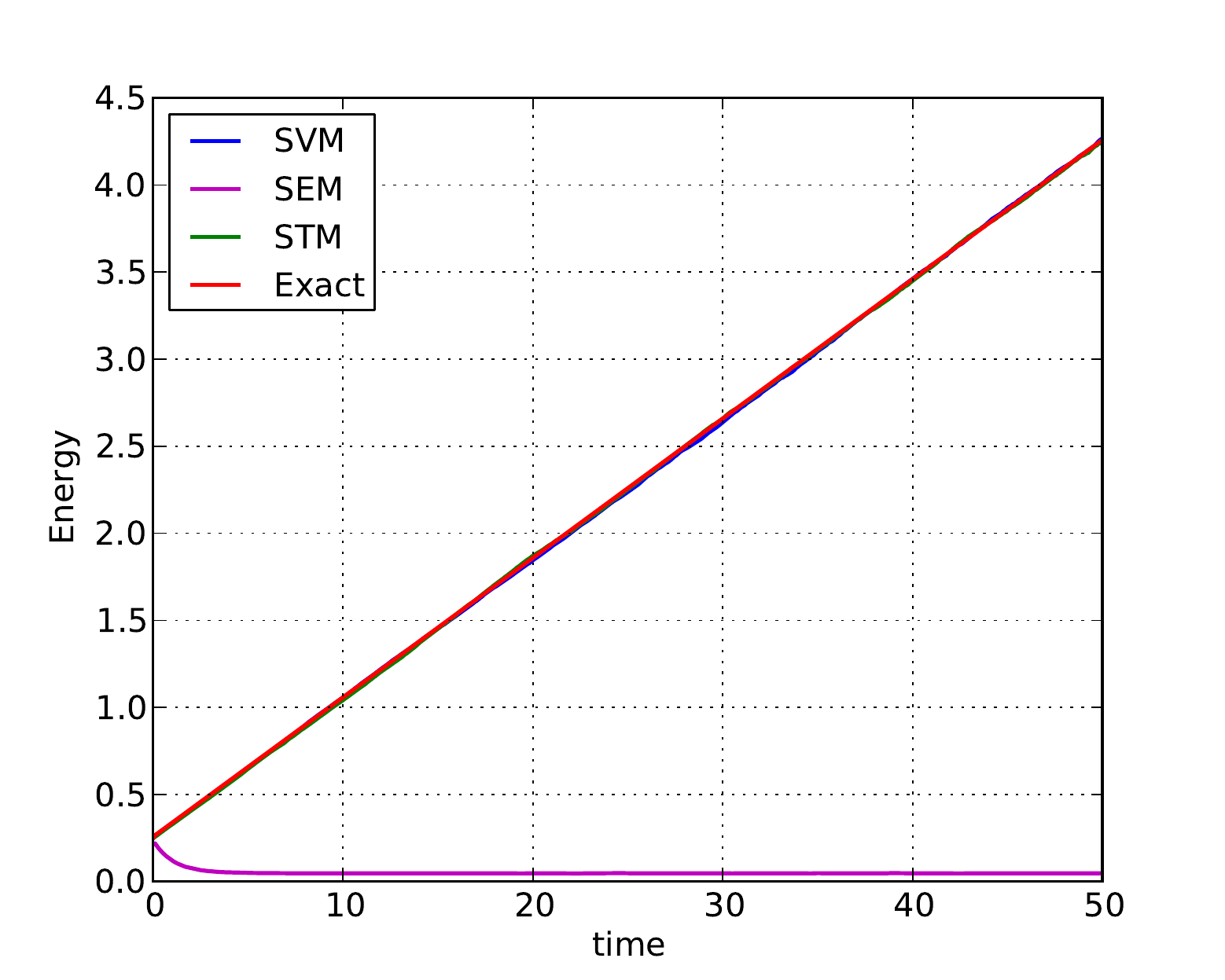}
\caption{}
\end{subfigure}
\caption{Plot (a) shows the expected values of the Hamiltonian $H$ of \eqref{eq:example1}  along the numerical solutions given by the three time stepping schemes presented above.
Plot (b) illustrates the expected values of the discrete energy ${\cal E}^n$  of the linear analogue of problem \eqref{eq:example1}.
 }
\label{fig:enresults}
\end{figure}

\bibliographystyle{siam}
\bibliography{biblio}
\end{document}